\documentclass[letterpaper]{amsart}

\usepackage{fullpage}
\usepackage[colorlinks]{hyperref}   

\usepackage{amsfonts,amsmath, amssymb,amsthm,amscd,mathdots}
\usepackage[usenames,dvipsnames]{color}
\usepackage{bm}
\usepackage{mathrsfs}
\usepackage{yfonts}
\usepackage{verbatim}
\usepackage{graphicx}
\usepackage[latin1]{inputenc}
\usepackage{latexsym}
\usepackage{lscape}
\usepackage{epsfig}
\usepackage{subfigure}
\include{bibtex}
\usepackage{setspace}

\begin{document}
\bibliographystyle{alpha}

\newcommand{\cross}{\begin{picture}(12,10)(-2,0)
\put(-8,0){\Large$\longrightarrow$}
\put(0,0){\Large$\downarrow$}
\end{picture}}

\newcommand{\cn}[1]{\overline{#1}}
\newcommand{\e}[0]{\epsilon}
\newcommand{\es}[0]{\emptyset}

\newcommand{\GUE}{{\rm GUE}}
\newcommand{\var}{{\rm var}}
\newcommand{\Prob}{{\rm Prob}}
\newcommand{\cdrp}{{\rm CDRP}}
\newcommand{\ddrp}{{\rm DDRP}}
\newcommand{\ZZ}{\ensuremath{\mathcal{Z}}}
\newcommand{\h}{\ensuremath{\mathcal{H}}}

\newcommand{\EE}{\ensuremath{\mathbb{E}}}
\newcommand{\PP}{\ensuremath{\mathbb{P}}}
\newcommand{\N}{\ensuremath{\mathbb{N}}}
\newcommand{\R}{\ensuremath{\mathbb{R}}}
\newcommand{\Rplus}{\ensuremath{(0,\infty)}}
\newcommand{\RNplus}{\ensuremath{(0,\infty)^N}}
\newcommand{\C}{\ensuremath{\mathbb{C}}}
\newcommand{\Z}{\ensuremath{\mathbb{Z}}}
\newcommand{\Q}{\ensuremath{\mathbb{Q}}}
\newcommand{\T}{\ensuremath{\mathbb{T}}}
\newcommand{\TN}{\ensuremath{\mathbb T_N}}
\newcommand{\TNstar}{\ensuremath{\mathbb T_N^*}}
\newcommand{\SN}[1]{\ensuremath{\mathbb Y_{#1}}}
\newcommand{\SNstar}[1]{\ensuremath{\mathbb Y^{*}_{#1}}}
\newcommand{\Yind}{\ensuremath{\mathcal{L}^{\rm{ind}}}}
\newcommand{\edge}{\textrm{edge}}
\newcommand{\E}[0]{\mathbb{E}}
\newcommand{\OO}[0]{\Omega}
\newcommand{\F}[0]{\mathfrak{F}}
\def \Ai {{\rm Ai}}
\newcommand{\G}[0]{\mathfrak{G}}
\newcommand{\ta}[0]{\theta}
\newcommand{\w}[0]{\omega}
\newcommand{\ra}[0]{\rightarrow}
\newcommand{\vectoro}{\overline}
\newtheorem{theorem}{Theorem}[section]
\newtheorem{partialtheorem}{Partial Theorem}[section]
\newtheorem{conj}[theorem]{Conjecture}
\newtheorem{lemma}[theorem]{Lemma}
\newtheorem{proposition}[theorem]{Proposition}
\newtheorem{corollary}[theorem]{Corollary}
\newtheorem{claim}[theorem]{Claim}
\newtheorem{experiment}[theorem]{Experimental Result}
\newtheorem{prop}{Proposition}
\numberwithin{equation}{section}

\def\todo#1{\marginpar{\raggedright\footnotesize #1}}
\def\change#1{{\color{green}\todo{change}#1}}
\def\note#1{\textup{\textsf{\color{blue}(#1)}}}

\theoremstyle{definition}
\newtheorem{rem}[theorem]{Remark}

\theoremstyle{definition}
\newtheorem{com}[theorem]{Comment}

\theoremstyle{definition}
\newtheorem{definition}[theorem]{Definition}
\newtheorem{example}[theorem]{Example}

\theoremstyle{definition}
\newtheorem{definitions}[theorem]{Definitions}

\newcommand{\be}{\begin{equation}}  \newcommand{\ee}{\end{equation}}
\def\e{\varepsilon}
\def\ind{\mathbf{1}}
\def\ddd{\displaystyle}
\newcommand{\nn}{\nonumber}
\providecommand{\abs}[1]{\vert#1\vert}
\providecommand{\norm}[1]{\Vert#1\Vert}
\providecommand{\pp}[1]{\langle#1\rangle}
\newcommand{\fl}[1]{\lfloor{#1}\rfloor}
\newcommand{\ce}[1]{\lceil{#1}\rceil}
\def\wt{\widetilde}  \def\wh{\widehat} \def\wb{\overline}
\def\zhat{\hat z}  

\newcommand\Laker[2]{\Lambda^{#1}_{#2}}  
\newcommand\Piker[2]{\Pi^{#1}_{#2}}  
\newcommand\Kker[2]{K^{#1}_{#2}}  

\newcommand\Mker[2]{M^{#1}_{#2}}  

\newcommand\Kbarker[2]{\bar K^{#1}_{#2}}  
\newcommand\Lker[2]{L^{#1}_{#2}}  
\newcommand\Pker[2]{P^{#1}_{#2}}  
\newcommand\Pbarker[2]{\bar P^{#1}_{#2}}  
\newcommand\Rker[2]{R^{#1}_{#2}}  
\newcommand\Dcon[2]{D^{#1}_{#2}}    

\newcommand\wfunc[2]{w^{#1}_{#2}}    

\newcommand\wf[2]{\Psi^{#1}_{#2}}    
\newcommand\awf[2]{\psi^{#1}_{#2}}    

\newcommand\thetar{\hat\theta}  
\newcommand\thetac{\theta}  
\newcommand\thetacv{\theta}  
\newcommand\thetarc{\gamma}  

\newcommand\proj{\phi}  

\newcommand\zarr{z}
\newcommand\zratio{\eta}  




\title[Tropical Combinatorics and Whittaker functions]{Tropical Combinatorics and Whittaker functions}

\author[I. Corwin]{Ivan Corwin}
\address{I. Corwin\\
Massachusetts Institute of Technology\\
Department of Mathematics\\
77 Massachusetts Avenue, Cambridge, MA 02139-4307, USA\newline
Clay Mathematics Institute\\ 10 Memorial Blvd. Suite 902\\ Providence, RI 02903, USA}
\email{ivan.corwin@gmail.com}
\author[N. O'Connell]{Neil O'Connell}
\address{N. O'Connell\\
  Mathematics Institute \\
  University of Warwick\\
  Coventry CV4 7AL, UK}
\email{N.M.O-Connell@warwick.ac.uk}
\author[T. Sepp\"{a}l\"{a}inen]{Timo Sepp\"{a}l\"{a}inen}
\address{T. Sepp\"{a}l\"{a}inen\\
  Department of Mathematics\\
  University of Wisconsin-Madison\\
  419 Van Vleck Hall\\
  Madison, WI 53706-1388, USA}
\email{seppalai@math.wisc.edu}
\author[N. Zygouras]{Nikolaos Zygouras}
\address{N. Zygouras\\
  Department of Statistics \\
  University of Warwick\\
  Coventry CV4 7AL, UK}
\email{N.Zygouras@warwick.ac.uk}
\date{\today}

\begin{abstract}

We establish a fundamental connection between the geometric RSK correspondence
and $GL(N,\R)$-Whittaker functions, analogous to the well known relationship between
the RSK correspondence and Schur functions.  This gives rise to a natural family of
measures associated with $GL(N,\R)$-Whittaker functions which are the analogues in
this setting of the Schur measures on integer partitions.  The corresponding analogue
of the Cauchy-Littlewood identity can be seen as a generalisation of an integral
identity for $GL(N,\R)$-Whittaker functions due to Bump and Stade.
As an application, we obtain an explicit integral formula for the Laplace transform of
the law of the partition function associated with a one-dimensional directed polymer
model with log-gamma weights recently introduced by one of the authors (TS).
\end{abstract}
\maketitle

\section{Introduction}

The Robinson-Schensted-Knuth (RSK) correspondence is a combinatorial mapping which plays a fundamental
role in the theory of Young tableaux, symmetric functions and representation theory.  In particular, it provides
a direct combinatorial proof of the Cauchy-Littlewood identity
$$\sum_\lambda s_\lambda(x) s_\lambda(y) = \prod_{i,j}(1-x_iy_j)^{-1},$$
where the sum is over integer partitions and $s_{\lambda}$ denotes the Schur function associated with the
partition $\lambda$.  The Schur function $s_\lambda(x)$ is a symmetric function in the variables
$x=(x_1,x_2,\ldots)$ defined by
$$s_\lambda(x)=\sum_{T} x^T$$
where the sum is over semistandard tableaux $T$ with shape $\lambda$ and $x^T=x_1^{\mu_1}x_2^{\mu_2}\ldots$,
where $\mu_i$ is the number of $i$'s in $T$.  For more background on symmetric functions we refer the reader
to \cite{m}.

The RSK mapping is defined by a combinatorial algorithm which associates to each matrix $M=\{m_{ij}\}$ with
non-negative integer entries a pair $(P,Q)$ of semi-standard tableaux with the same shape.  By a version of
Greene's theorem~\cite{g,Kir}, it can also be defined via expressions in the $(\max,+)$ semi-ring. This was extended to matrices with
real entries by Berenstein and Kirillov~\cite{bki}.  Replacing these expressions by their analogues in the usual $(+,\times)$ algebra,
A.N. Kirillov~\cite{Kir} introduced a geometric lifting of the Berenstein-Kirillov correspondence which he called
the `tropical RSK correspondence', in honour of M.-P. Sch\"utzenberger (1920--1996).
However, for many readers nowadays the word `tropical' indicates just the opposite,
so to avoid confusion we will refer to Kirillov's construction as the {\em geometric} RSK correspondence,
as in the theory of {\em geometric crystals}~\cite{BK,BK2}, which is closely related. This correspondence has been studied further from
a dynamical point of view by Noumi and Yamada~\cite{NY} and related to Dodgson's condensation method for computing determinants in~\cite{dk}.

It is natural to ask if there is an analogue of the Cauchy-Littlewood identity for the geometric RSK correspondence.
A priori it is not at all clear that this should exist:  something quite remarkable needs to happen in order to see
the product structures on both sides which characterise this identity.   In this paper we show that in fact
these product structures do appear in the correct formulation and,
moreover, the role of the Schur functions is now played by $GL(N,\R)$-Whittaker functions.  In a particular case
(corresponding to square matrices) the analogue of the Cauchy-Littlewood identity is in fact a well known
integral identity for $GL(N,\R)$-Whittaker functions due to Bump and Stade~\cite{B,St,GLO}.

Before stating our main results, we will first explain the origin of these product structures in the context of the
RSK correspondence and its role in the proof of the Cauchy-Littlewood identity.  The RSK mapping associates
to each matrix $M=\{m_{ij}\}$ with non-negative integer entries a pair $(P,Q)$ of semi-standard tableaux with
the same shape.  If $M$ is an $n\times N$ matrix then $P$ has entries from $\{1,2,\ldots,N\}$ and $Q$ has entries
from $\{1,2,\ldots,n\}$.  Moreover, the vector $(R_1,\ldots,R_n)$ of row sums of $M$ is the type of $P$, that is,
$R_1$ is the number of $1$'s in the semistandard tableau $P$, $R_2$ is the number of $2$'s, and so on.
Similarly, the vector of column sums $(C_1,\ldots,C_N)$ is the type of $Q$.  Note that this reflects the well
known symmetry property of RSK, namely that if $M\mapsto (P,Q)$ then $M^t\mapsto (Q,P)$.  It follows that
$$\prod_{i,j} (x_i y_j)^{m_{ij}} = \prod_i x^{R_i}\prod_j y^{C_j} = x^P y^Q ,$$
and summing both sides gives the Cauchy-Littlewood identity:  on the left we sum over all matrices with
non-negative integer entries and on the right we sum over all pairs of semistandard tableaux $P$ and $Q$
with the same shape.

Another way of interpreting the above argument is as follows.  Let $p_1,\ldots,p_n$ and $q_1,\ldots,q_N$ be a collection of positive numbers such that $0<p_iq_j<1$ for all $i$ and $j$,
and consider the probability measure on $n\times N$ matrices defined by
$$P(\{M\})= \prod_{i,j}(1-p_iq_j) \prod_{i,j} (p_i q_j)^{m_{ij}}.$$
From the above discussion, the push-forward of this probability measure
onto the shape of the tableaux obtained under the RSK mapping is given by
\begin{eqnarray}\label{OkSchur}
\tilde P(\{\lambda\})= \prod_{i,j}(1-p_iq_j) s_\lambda(p) s_\lambda(q).
\end{eqnarray}
Now, the Cauchy-Littlewood identity is essentially equivalent to the fact that $\tilde P$ is a probability measure on the set
of integer partitions.  Such probability measures are known as {\em Schur measures} \cite{OkInfWedge}.

We now turn to the geometric RSK correspondence.  The input is an $n\times N$ matrix $X=\{x_{ij}\}$ with strictly positive real entries and, supposing here for convenience that $n\ge N$, the analogue of the $P$-tableau is a triangular array $(z_{k,\ell})_{1\le \ell\le k\le N}$ of non-negative real numbers.  The `shape' of $P$ is the vector $z_{N,\cdot}$.
Let $\hat\theta_1,\ldots,\hat\theta_n$ and $\theta_1,\ldots,\theta_N$ be a collection of real numbers satisfying
$\hat\theta_i+\theta_j>0$ for all $i$ and $j$, and consider the product measure on input matrices $X$ defined by
$$\mu(dX)= \prod_{i,j} \nu_{\hat\theta_i+\theta_j} (dx_{ij}),$$
where $\nu_\theta$ denotes the distribution of the inverse of a Gamma random variable
$$\nu_\theta(dx)= \frac{1}{\Gamma(\theta)} x^{-\theta-1}\exp\left\{-\frac{1}{x}\right\} dx.$$
The main result of this paper (Theorem \ref{main corollary}) is that the push-forward of the probability measure $\mu$
under the geometric RSK mapping onto the `shape' $(z_{N,\ell})_{1\leq \ell\leq N}$ is given by
\begin{eqnarray}\label{intro9}
\mu_n^N(dy)=
 \Psi_\theta(y)
\int_{\iota \mathbb{R}^N} d\lambda \, s_N(\lambda)
 \Psi_{-\lambda}(y)
 \prod_{m=1}^n\prod_{i=1}^N\frac{\Gamma(\hat\theta_m+\lambda_i)}{\Gamma(\theta_i+\hat\theta_m)}\,\,\,\prod_{i=1}^N \frac{dy_i}{y_i},
\end{eqnarray}
where the functions $\Psi_{\nu}(y)$ are $GL(N,\R)$-Whittaker functions (defined in Section \ref{WhitFUnSect} below)
and $$s_N(\lambda)=\frac{1}{(2\pi \iota)^N N!} \prod_{j\ne k} \Gamma(\lambda_j-\lambda_k)^{-1}.$$
The probability measure $\mu_n^N$ is the analogue of the Schur measure in this setting and shall be referred to
as a {\em Whittaker measure}; the fact that it integrates to one is the analogue of the Cauchy-Littlewood identity.

In the special case $N=n$, we obtain the following simplification:
\begin{equation}\label{se}
\mu_N^N(dy)=\prod_{m=1}^N\prod_{i=1}^N\Gamma(\theta_i+\hat\theta_m)^{-1} e^{-y_N^{-1}}
\Psi_\theta(y) \Psi_{\hat\theta}(y)\prod_{i=1}^N\frac{dy_i}{y_i}.
\end{equation}
In this case, the analogue of the Cauchy-Littlewood identity reduces to
$$\int_{\R_+^N} e^{-y_N^{-1}} \Psi_\theta(y) \Psi_{\hat\theta}(y)\prod_{i=1}^N\frac{dy_i}{y_i}
= \prod_{m=1}^N\prod_{i=1}^N\Gamma(\theta_i+\hat\theta_m),$$
which can be seen to be equivalent to a Whittaker integral identity due to Bump and Stade~\cite{B,St,GLO}.
We remark that, while this gives a new interpretation of the Bump-Stade Whittaker integral identity, we do
not give a new proof of this identity in the present paper; in fact, we use it to obtain the simplification~(\ref{se}).

We prove Theorem \ref{main corollary}  by considering a dynamical version of the geometric RSK construction due to Noumi and Yamada~\cite{NY}, allowing $n$ to increase as we successively add rows to the input matrix $X$. The image triangular array $z(n)$ evolves as a Markov process in discrete time $n$  subject to a particular entrance law for $n<N$. First we prove that the shape $z_{N,\cdot}(n)$ evolves marginally as a Markov process in its own filtration.  The Markov property of $z_{N,\cdot}(n)$ relies (via the theory of Markov functions) on an algebraic intertwining relation between Markov kernels for $z(n)$ and $z_{N,\cdot}(n)$ as well as on a  limiting argument which shows that the result holds for the particular entrance law for $z(n)$ dictated by  the geometric RSK correspondence. Secondly we prove that this Markov process can be diagonalized in terms of $GL(N,\R)$-Whittaker functions. This yields the formula \eqref{intro9}
for the fixed $n$ probability distribution of the shape.

Our results have an important application to the study of directed polymers, analogous to the role of the RSK correspondence
in the study of longest increasing subsequences and last passage percolation~\cite{AD,BDJ,KJ,KJ2,LS,VK}.
According to the definition of geometric RSK,
we can write $z_{N,1}(n) = \sum_{\pi}\prod_{(i,j)\in\pi} x_{ij}$ where the summation is over all 'up/right' lattice paths in $\Z^2$
from $(1,1)$ to $(n,N)$.  Under the measure $\mu$, this random variable can be interpreted as a partition function for a
directed polymer in a random environment given by the weights $x_{ij}$.  For a particular (homogeneous) choice of the
parameters $\hat\theta_i+\theta_j=\gamma$ for all $i$ and $j$, this model was introduced and studied in the paper~\cite{S}.
In particular, setting $Z_n=z_{n,1}(n)$, it was shown in \cite{S} that the free energy is given explicitly by
 \begin{equation}
\lim_{n\to \infty} \frac1n\log Z_n = -2\Psi(\gamma/2)
\end{equation}
almost surely, where $\Psi(x)=[\log \Gamma]'(x)$ is the digamma function, and moreover
\begin{equation}
 \limsup_{n\to \infty} \frac{\var \log Z_n}{n^{2/3}} <\infty .
\end{equation}
These asymptotic results were obtained via the observation that for this polymer model
there is an analogue of the output theorem (or Burke property) for the single server queue.  This observation, combined
with recent progress on a related model in~\cite{OCon}, provided the inspiration for the present work.   In fact, the model
considered in \cite{OCon} is a degeneration of the one we presently consider (c.f. Section \ref{semi-disc}).

Using our main result, we obtain the following explicit formula for the Laplace transform of the law of the partition function $z_{N,1}(n)$
under the measure $\mu$ (this statement is also contained in Theorem \ref{main corollary}):
\begin{equation}\label{fabove}
\EE \left[e^{-s z_{N,1}(n)}\right] =
  \int_{\iota \mathbb{R}^N}d\lambda\, s^{\sum_{i=1}^N(\theta_i-\lambda_i)}
          \prod_{1\leq i,j\leq N}\Gamma(\lambda_i-\theta_j) \,\prod_{m=1}^n\prod_{i=1}^N\frac{\Gamma(\lambda_i+\hat\theta_m)}{\Gamma(\theta_i+\hat\theta_m)} s_N(\lambda),
\end{equation}
where the poles of the functions $\Gamma(\lambda_i-\theta_j)$ and $\Gamma(\lambda_i+\hat\theta_m)$ are not encountered as we may assume without loss of generality that $\thetar_m>0$ for all $m$ and $\thetac_j<0$ for all $j$.

Our formula (\ref{fabove}) has recently been applied in \cite{BCR} to prove the following asymptotic result:
there exists $\gamma^*>0$ such that the inverse-gamma weight polymer free energy with parameter
$\gamma\in (0,\gamma^*)$ has limiting fluctuation distribution given by
\begin{equation}
\lim_{n\to \infty} \PP\left(\frac{ \log Z_n - n\bar{f}_{\gamma}}{n^{1/3}}\leq r\right) = F_{{\rm GUE}}\left(\left(\frac{\bar{g}_{\gamma}}{2}\right)^{-1/3} r\right)
\end{equation}
where $\bar{f}_{\gamma} = -2\Psi(\gamma/2)$, $\bar{g}_{\gamma} = -2 \Psi''(\gamma/2)$ and $F_{{\rm GUE}}$ is the GUE Tracy-Widom distribution function.  The restriction on the parameter $\gamma$ is present for purely technical reasons.

There have been many recent developments on the Kardar-Parisi-Zhang (KPZ) equation~\cite{ICreview,FSreview,ss}
and various discretisations of this equation which have an underlying algebraic structure.
The latter include the Whittaker measures and processes introduced in the present paper, and a particular family
of degenerations of these which were introduced earlier in \cite{OCon}.  These in turn can be seen as degenerations of the
Macdonald measures and processes introduced in \cite{BorCor} which gives access to the theory of Macdonald polynomials
and incorporates a wide range of other interesting degenerations including $q$-deformations of the Whittaker measures and
processes discussed here.  A feature which is important for the Whittaker measures we consider here is the fundamental
link to the geometric RSK correspondence.  In particular, our main result is not a degeneration of any of the main results in the
more recent paper \cite{BorCor}.

Further developments appear in the recent work \cite{OSZ} which explores this connection from a combinatorial point of view, in contrast to the dynamical approach of the present work. The paper \cite{OSZ} provides further insight into the appearance of Whittaker functions in this setting and, among other things, 
studies the restriction of the geometric RSK mapping to symmetric matrices.

In another direction, it is possible
to define an analogue of the geometric RSK mapping directly in the continuum setting of the KPZ equation, with input given
by a two-dimensional strip of space-time white noise~\cite{ow}.  Some recent progress towards understanding the law of the
analogue of the `shape' in this setting has been made in \cite{BorCor,ch}.


The outline of the paper is as follows.
In Section \ref{Kirillov corres} we introduce the geometric RSK correspondence via two equivalent approaches: row insertion procedure and non-intersecting lattice paths.
Section \ref{raninputsec} provides the main set of results of this paper. We state the main algebraic content of the paper in the form of the intertwining relation of
Proposition \ref{itpr1}; we define Whittaker functions; we state the paper's main results -- Theorems \ref{markovprojthm} and \ref{main corollary}). Within that section we also record the invariant distribution of dynamical geometric RSK correspondence, and explain the connection Pitman's $2M-X$ theorem. Section \ref{scalinglimits} details how in certain scaling limits of our work, one recovers previously discovered results. Proofs of our main results are contained in Section \ref{proofsec}.

\bigskip

\noindent {\bf Acknowledgements.}
We thank Jinho Baik, Gerard Ben Arous, Philippe Biane, Alexei Borodin, Percy Deift, Jeremy Quastel, Pierre van Moerbeke, Herbert Spohn, Craig Tracy and Lauren Williams for helpful discussions, and acknowledge MSRI, IMPA, Mathematisches Forschungsinstitut Oberwolfach and the University of Warwick (with financial support from grants EP/I014829/1 and IRG-246809) for hospitality during this project. IC is partially supported by the NSF through PIRE grant OISE-07-30136 and DMS-1208998 as well as by Microsoft Research through the Schramm Memorial Fellowship, and by the Clay Mathematics Institute through a Clay Research Fellowship. NO'C is partially supported by EPSRC grant EP/I014829/1. TS is partially supported by National Science Foundation grant DMS-1003651 and by the Wisconsin Alumni Research Foundation. NZ is partially supported by a Marie Curie International Reintegration Grant within the 7th European Community Framework Programme, IRG-246809.

\section{Geometric RSK correspondence}\label{Kirillov corres}

In this section we describe an extension of Kirillov's `tropical' (or presently `geometric') RSK correspondence \cite{Kir} to rectangular matrices. We follow mainly the development in  \cite{NY} but with a slightly different convention for  indices. We describe first a geometric row insertion procedure, and then expand this into a procedure for  inserting a word into a triangular array (see Example \ref{example1} for a step-by-step illustration of these procedures).   Repeated insertions create a temporal evolution of the array.   In addition to insertion into an already existing array we consider insertion into an initially empty array.  This latter version will have an equivalent description in terms of weights of configurations  of lattice paths.

\subsection{Geometric RSK via row insertion}

\begin{definition}\label{NYdef}
Let $1\le \ell\le N$. Consider two {\it words} $\xi=(\xi_{\ell},\ldots,\xi_N)$ and $b=(b_\ell,\ldots, b_N)$ with strictly positive real entries. {\it Geometric row insertion} of the word $b$ into the word $\xi$    transforms  $(\xi,b)$ into a new pair $(\xi',b')$ where $\xi'=(\xi'_{\ell},\ldots, \xi'_N)$ and $b'=(b'_{\ell+1},\ldots, b'_N)$. The transformation is notated and defined as follows:
\begin{equation}
\begin{array}{ccc}
& b & \\
\xi & \cross & \xi' \\
& b'
\end{array}
\quad
\textrm{where} \quad
\begin{cases}
\xi'_\ell=b_\ell\xi_\ell, &\\[4pt]
\xi'_k =b_k(\xi_{k-1}' + \xi_k), & \ell+1\le k\le N\\[5pt]
b'_{k} = b_k \ddd\frac{ \xi_{k}\xi'_{k-1}}{\xi_{k-1}\xi'_k}, & \ell+1\leq k\leq N.
\end{cases}
\label{g-row-ins}\end{equation}
If $\ell=N$ output $b'$ is empty and we write $b'=\es$.  
In addition to $\xi\in(0,\infty)^{N-\ell+1}$ we admit the case
$\xi=(1,0,\dotsc,0)$.
  This will correspond to row insertion into an initially empty word.
With  $e^{(k)}_1=(1,0,\dotsc,0)$ denoting the first unit $k$-vector,
the notation and  definition are now
\begin{equation}
\begin{array}{ccc}
& b & \\
e_1^{(N-\ell+1)} & \cross & \xi'  \\[3pt]
\end{array}
\quad
\textrm{where} \quad  \xi'_{k}=\prod_{i=\ell}^{k} b_{i},  \quad \ell\le k\le N.
\label{es-row-ins}\end{equation}
This is consistent with \eqref{g-row-ins} except that   output $b'$ is not defined and hence
not displayed in the diagram above.
\end{definition}

The next step is
geometric row insertion of a word into a  triangular array.
For $N\ge 1$ let $\TN = (\zarr_{k\ell}: 1\le \ell\le k\le N, \textrm{and } \zarr_{k\ell}\in(0,\infty))$, i.e. the set of triangular arrays with positive real entries.   The bottom picture of Figure \ref{pathfig} illustrates
an element of $\T_5$.  $(z_{k\ell})$ consists of rows indexed by $k$ and southeast-pointing
diagonals indexed by $\ell$.

\begin{definition}\label{NYdef-z}
  Given   $\zarr\in \TN$ and a word $b\in(0,\infty)^N$.  Geometric row insertion of   $b$ into   $\zarr$    outputs a new triangular array $\zarr'\in \TN$. This procedure is denoted by
\begin{equation}
\zarr' = \zarr\leftarrow b
\end{equation}
and it consists of $N$ iterations of the basic row insertion. 
  For $1\leq \ell\leq N$ form words $\zarr_\ell = (\zarr_{\ell\ell},\ldots, \zarr_{N\ell})$.
Begin by setting   $a_1=b$. Then for $\ell=1,\ldots , N$  recursively apply the map
\begin{equation}\label{NYalg}
\begin{array}{ccc}
& a_\ell & \\
\zarr_\ell & \cross & \zarr'_\ell \\
& a_{\ell+1}
\end{array}
\end{equation}
from Definition  \ref{NYdef},
where $a_{\ell+1}=a'_\ell$.  The last output $a_{N+1}$ is empty.
The new array  $\zarr'=(\zarr'_{k\ell}: 1\le \ell\le k\le N)$ is formed from the
words $\zarr'_\ell = (\zarr'_{\ell\ell},\ldots, \zarr'_{N\ell})$.
  Along the way the procedure  constructs an auxiliary triangular array $a=(a_{k\ell}: 1\le \ell\le k\le N)$  with diagonals $a_\ell=(a_{\ell\ell},\dotsc, a_{N\ell})$.
\end{definition}

\begin{figure}
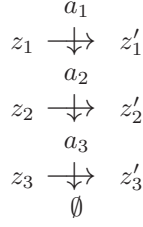

\begin{equation*}
\begin{array}{ccc}
& a_1 & \\
\zarr_1 & \cross & z'_1 \\
& a_2 &\\
\zarr_2 & \cross & z'_2 \\
& a_3 &\\
\zarr_3 & \cross & z'_3 \\ &\emptyset&
\end{array}
\end{equation*}
\caption{Illustration of $\zarr' = \zarr\leftarrow a_1$  when $N=3$. Geometric row insertion of the  word $a_1=(a_{11},a_{21},a_{31})$ into the triangular array $\zarr$ is defined recursively by insertion of $a_i$ into $\zarr_i$ with outputs $z_{i}'$ and $a_{i+1}$.  After step 3 the process
has been exhausted:  $a_3=(a_{33})$ has one entry and $a_4$ is an empty vector. }
\label{NYfiga}
\end{figure}

Definition \ref{NYdef-z} of $\zarr' = \zarr\leftarrow b$ can be summarized by these equations:
\be\begin{aligned}\label{NYalg1}
a_{k,1}&=b_k  &\textrm{for } 1\le k\leq N\\
a_{k+1,\ell+1}&=a_{k+1,\ell} \frac{\zarr_{k+1,\ell} \zarr_{k,\ell}'}{\zarr_{k+1,\ell}' \zarr_{k,\ell}}
&\textrm{for } 1\le \ell\leq k<N\\
\zarr_{k,\ell}'&= a_{k,\ell}(\zarr_{k,\ell}+\zarr_{k-1,\ell}')&\textrm{for } 1\le \ell<k\leq N\\
\zarr_{k,k}'&= a_{k,k}\zarr_{k,k} &\textrm{for } 1\le k\leq N.
\end{aligned}\ee
This procedure is illustrated in Figure \ref{NYfiga} when $N=3$.

Iteration of the insertion procedure defines a temporal evolution $\zarr(n)$, $n=0,1,2,\dotsc$, of
an array $\zarr(n)\in\TN$.  This evolution is driven by a  semi-infinite   matrix
$d=(d_{nj}: n\ge 1, 1\leq j\leq N)$ of positive weights $d_{nj}$.
We write
$d^{[1,n]}=(d_{ij}: 1\leq i\leq n, 1\leq j\leq N)$
for the matrix   of the first $n$ rows of $d$,
and  $d^{[n]}=(d_{n1},\dotsc, d_{nN})$ for the $n^{\text{th}}$ row of $d$.   The
temporal evolution is then defined by successive insertions of rows of $d$ into
the initial array:  given $\zarr(0)\in\TN$ and $d$,  then iteratively for  $n\ge 1$,
\be
{\zarr}(n)  = \Bigl[  \zarr(n-1)  \leftarrow d^{[n]} \Bigr]
= \Bigl[  \zarr(0) \leftarrow d^{[1]}\leftarrow d^{[2]}\leftarrow \cdots \leftarrow d^{[n]} \Bigr].
  \label{z(n)}\ee
Figure \ref{NYfig3} illustrates.

 \begin{figure}
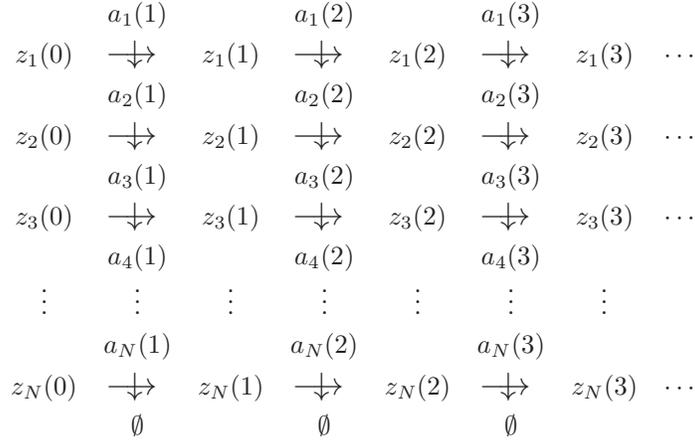

\begin{equation*}
\begin{array}{cccccccc}
           &a_1(1) &           & a_1(2)     &        & a_1(3) & &      \\[2pt]
\zarr_1(0)        &\cross         &\zarr_1(1)     &\cross              & \zarr_1(2) &\cross          & \zarr_1(3)   & \cdots\\[3pt]
           &a_2(1)           &           & a_2(2)             &        & a_2(3)         & &      \\[2pt]
\zarr_2(0)         &\cross         &\zarr_2(1)        &\cross              & \zarr_2(2) &\cross          & \zarr_2(3)   & \cdots\\[3pt]
           &a_3(1)            &           & a_3(2)              &        & a_3(3)         & &      \\[2pt]
\zarr_3(0)       &\cross         &\zarr_3(1)         &\cross              & \zarr_3(2)     &\cross          & \zarr_3(3)   & \cdots\\[3pt]
     &a_4(1)            &           & a_4(2)              &        & a_4(3)         & &      \\
  \vdots         &\vdots           & \vdots          & \vdots               & \vdots       & \vdots           &\vdots & \\[3pt]
               &a_N(1)            &           & a_N(2)              &        & a_N(3)         & &      \\[2pt]
\zarr_N(0)       &\cross         &\zarr_N(1)         &\cross              & \zarr_N(2)     &\cross          & \zarr_N(3)   & \cdots\\[3pt]
               &\es           &           &\es              &        & \es       & &      \\
\end{array}
\end{equation*}
\caption{Evolution of the array $\zarr(n)$ in  state space $\TN$ over time $n=0,1,2,\dotsc$
   The initial state
$\zarr(0)$ is on the left edge.   The inputs
come from the $d$-array: $a_1(n)=d^{[n]}=(d_{n,1},\dotsc, d_{n,N})$. }
 \label{NYfig3}
\end{figure}


\begin{example}\label{example1}
Let us illustrate  the   construction with a step-by-step  example.  (We use rational numbers, though the procedures apply more generally for non-negative reals.) First consider geometric row insertion of $b=(1,5)$ into $\xi = (3,2)$. By the procedure of (\ref{g-row-ins}) we first compute the output $\xi'_1 = b_1 \xi _1 = 1 \cdot 3 = 3$ and then $\xi'_2 = b_2 ( \xi'_1 + \xi_2) = 5(3+2) = 25$. Then we compute $b'_2 =  b_2 \frac{\xi_2 \xi'_1}{\xi_1 \xi'_2} = 5 \frac{ 2 \cdot 3}{3 \cdot 25} = \frac{2}{5}$. This calculation is represented by the diagram
\begin{equation*}
\begin{array}{ccc}
& (1,5) & \\
(3,2) & \cross & (3,25) .\\
& (\tfrac{2}{5}) &
\end{array}
\end{equation*}

We now illustrate geometric insertion of $a_1=(2,2,4)$ into the triangular array $z$ with $z_1=(4,1,3)$, $z_2= (3,7)$ and $z_3=(2)$. That is to say, that $z$ is given by
\begin{equation*}
\begin{array}{ccccccccc}
&& 4 &&\\
\\
& 3 && 1 &\\
\\
2 && 7 && 3\\
\end{array}
\end{equation*}

This insertion is performed by completing the following diagram.
\begin{equation*}
\begin{array}{ccc}
              &(2,2,4)                        &                               \\[2pt]
(4,1,3)       &\cross                         & z'_1                          \\[2pt]
              &  a_2                          &                               \\[2pt]
(3,7)         &  \cross                       & z'_2                          \\[2pt]
              &   a_3                         &                               \\[2pt]
 (2)          &  \cross                       & z'_3                          \\[2pt]
\end{array}
\end{equation*}

After calculating each insertion, starting from the top and sequentially going down, we arrive at  the following diagram.
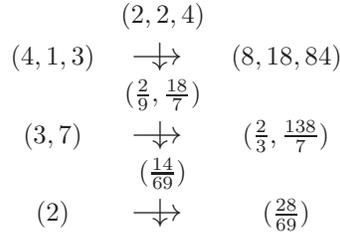
\begin{equation*}
\begin{array}{ccc}
              &(2,2,4)                        &                               \\[2pt]
(4,1,3)       &\cross                         & (8,18,84)                     \\[2pt]
              &  (\tfrac{2}{9},\tfrac{18}{7}) &                               \\[2pt]
(3,7)         &  \cross                       & (\tfrac{2}{3},\tfrac{138}{7}) \\[2pt]
              &   (\tfrac{14}{69})            &                               \\[2pt]
 (2)          &  \cross                       & (\tfrac{28}{69})              \\[3pt]
\end{array}
\end{equation*}
From the right-hand side of the diagram we read off the new array
\begin{equation*}
z' \; \;=\;  \quad
\begin{array}{ccccccccc}
&& 8 &&\\
\\
& \tfrac{2}{3} && 18 &\\
\\
\tfrac{28}{69} && \tfrac{138}{7} && 84\\
\end{array}
\end{equation*}
thus completing the insertion procedure.
\end{example}

 \begin{figure}
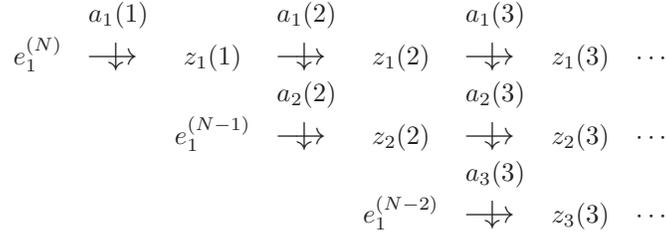

\begin{equation*}
\begin{array}{cccccccc}
           &a_1(1) &           & a_1(2)     &        & a_1(3) & &      \\[2pt]
e^{(N)}_1       &\cross         &\zarr_1(1)     &\cross              & \zarr_1(2) &\cross          & \zarr_1(3)   & \cdots\\[2pt]
           &          &           & a_2(2)             &        & a_2(3)         & &      \\[2pt]
       &         &e^{(N-1)}_1        &\cross              & \zarr_2(2) &\cross          & \zarr_2(3)   & \cdots\\[2pt]
           &             &           &                 &        & a_3(3)         & &      \\[2pt]
        &       &         &               &e^{(N-2)}_1    &\cross          & \zarr_3(3)   & \cdots\\[2pt]
\end{array}
\end{equation*}
\caption{Evolution of the array $\zarr(n)$  started from the empty array
 $\zarr(0)=\es$.   $e^{(k)}_1$ represents the word $(1,0,\dotsc,0)$ of length $k$.  By  Proposition \ref{NYrecprop}  $\zarr(n)$  is equal  to the image $P_{n,N}(d^{[1,n]})$  of the weight matrix
  $d^{[1,n]}=\{a_1(i)\}_{i=1}^{n}$ under the geometric RSK correspondence.}\label{NYfigb}
\end{figure}

Finally we consider the insertion process with an empty initial array. $N$ is still the
fixed size parameter of the array.  Initially $z(0)$ is empty which we denote by
$z(0)=\es$.   The array grows by adding one new diagonal $z_\ell$ at each
time.   At time   $n\in\{1,\dotsc, N\}$,
the already existing diagonals $z_1,\dotsc, z_{n-1}$ are updated by
inserting $a_1(n)$  and   iterating
  as in \eqref{NYalg}, and a new diagonal $z_n$ is filled by inserting
 $a_{n}(n)$  according to \eqref{es-row-ins}.
Consequently,  at time $1\le n<N$,
 the currently defined array with strictly positive entries  is
$\zarr(n)=\{\zarr_{k\ell}(n):\ 1\le k\le N,\, 1\le \ell\le k\wedge n\}$.
We consider the entries  $\{\zarr_{k\ell}(n):\ n< \ell\le k\le N\}$ undefined.
 At time $n=N$ the
array is full, and after time $N$ its evolution continues according to \eqref{z(n)}.
 The evolution
of $\zarr(n)$ from $\zarr(0)=\es$ is illustrated by Figure \ref{NYfigb}.

\begin{rem}   Instead of having truncated  arrays in the evolution
$\{\zarr(n): 0\le n<N\}$ from  $\zarr(0)=\es$, we could also choose to fill the undefined
portion of the array
 with certain conventions that  are consistent with the
update rules. This would include use of `singular values' $0$ and $\infty$.
For example, at time  $0\le n<N$, diagonal $z_{n+1}(n)$ would equal
$(1,0,\dotsc,0)$, in accordance with \eqref{es-row-ins}.   State space
$\TN$ would be replaced with a larger space $\TNstar$ that contains $\es$
and these other partially  singular arrays. In this paper we will not use these conventions.
\label{sing-rem}\end{rem}

\subsection{Geometric RSK via non-intersecting lattice paths}

We turn to  an alternative definition of the evolution in Figure \ref{NYfigb}
in terms of configurations of non-intersecting lattice paths.
As before $N\ge 1$ is fixed and
the input of the process  is the semi-infinite  matrix
$d=(d_{ij}: i\ge 1,  1\leq j\leq N)$ of positive real weights.
For each $n\ge 1$   form  the  $n\times N$ matrix
$d^{[1,n]}=(d_{ij}: 1\le i\le n,  1\leq j\leq N)$.
  For $1\leq \ell \leq k\le N$ let $\Pi_{n,k}^{\ell}$ denote the set of $\ell$-tuples $\pi=(\pi_1,\ldots, \pi_\ell)$ of non-intersecting lattice paths in $\Z^2$ such that, for $1\le r\le \ell$,  $\pi_{r}$ is a lattice path from $(1,r)$ to $(n,k+r-\ell)$. A `lattice path' only takes unit steps in the coordinate directions between nearest-neighbor lattice points of $\Z^2$ (i.e., up or right); non-intersecting means that paths do not touch. The {\sl weight} of an $\ell$-tuple $\pi = (\pi_1,\ldots, \pi_\ell)$ of such paths is
\begin{equation}
wt(\pi) = \prod_{r=1}^{\ell} \prod_{(i,j)\in \pi_{r}} d_{ij}.
\end{equation}
  For   $1\le \ell\le k\le N$   let
\be \tau_{k,\ell}(n) = \sum_{\pi\in\Pi_{n,k}^{\ell}} wt(\pi).   \label{tau} \ee
For $0\leq n<\ell< k\leq N$  the set of paths $\Pi_{n,k}^\ell$ is empty and we
take the empty sum to equal zero.
At $\ell=k$ there is a unique $\ell$-tuple, and in fact we have the equation
\begin{eqnarray*}
\tau_{k,\ell}(n)=\delta_{k,\ell}\,\tau_{k,n}(n)\qquad\text{for}\qquad  0\leq n< \ell\leq k\leq N
\end{eqnarray*}
where  $\delta_{k,\ell}$ is the Kronecker delta.  For $\ell=0$ the right
convention turns out to be $\tau_{k,0}(n)=1$ for $1\le k\leq N$.

The  array $\zarr(n)=\{\zarr_{k,\ell}(n):\ 1\le k\le N,\ 1\le \ell\le k\wedge n\}$  is now defined by
\begin{equation}
\zarr_{k,1}(n)\cdots \zarr_{k,\ell}(n) = \tau_{k,\ell}(n).
\label{ztau-3}\end{equation}
The elements  $\big(\zarr_{k\ell}(n):\ n< \ell\le k\le N\big)$ we regard as undefined, even though
strictly speaking one more element, namely $\zarr_{n+1,n+1}(n)$, could be consistently defined as $1$.   In the spirit of Remark \ref{sing-rem} we could also  replace the
undefined array elements with particular singular values. See Figure \ref{pathfig} for an illustration.

\begin{figure}
\begin{equation*}
\begin{array}{ccccccc}
d_{15} && d_{25} && d_{35} && d_{45}\\
       &&        &&  &&\uparrow\\
d_{14} && d_{24} && d_{34} && d_{44}\\
       &&        &&  &&\uparrow\\
d_{13} && d_{23} &\rightarrow& d_{33} &\rightarrow & d_{43}\\
       &&        \uparrow&&  &&\\
d_{12} && d_{22} && d_{32} && d_{42}\\
       &&        \uparrow&&  &&\\
d_{11} &\rightarrow& d_{21} && d_{31} && d_{41}
\end{array}
\qquad \qquad\qquad
\begin{array}{ccccccc}
d_{15} && d_{25} && d_{35} &\rightarrow& d_{45}\\
       &&        && \uparrow && \\
d_{14} &\rightarrow& d_{24} &\rightarrow& d_{34} && d_{44}\\
  \uparrow     &&        &&  &&\uparrow\\
d_{13} && d_{23} &\rightarrow& d_{33} &\rightarrow & d_{43}\\
 \uparrow      &&        \uparrow&&  &&\\
d_{12} && d_{22} && d_{32} && d_{42}\\
       &&        \uparrow&&  &&\\
d_{11} &\rightarrow& d_{21} && d_{31} && d_{41}
\end{array}
\end{equation*}
\\[6pt]
\begin{equation*}
\begin{array}{ccccccccc}
&&&& \zarr_{11}(4) &&&&\\
\\
&&& \zarr_{22}(4) && \zarr_{21}(4) &&&\\
\\
&&\zarr_{33}(4) && \zarr_{32}(4) && \zarr_{31}(4) &&\\
\\
&\zarr_{44}(4) && \zarr_{43}(4) && \zarr_{42}(4) && \zarr_{41}(4)&\\
\\
\mathbf{\zarr_{55}(4)} && \zarr_{54}(4) && \zarr_{53}(4) && \zarr_{52}(4)&& \zarr_{51}(4)
\end{array}
\end{equation*}
\caption{Illustration of the path construction  for $n=4$ and $N=5$. Note that the matrix $d$ is presented in Cartesian coordinates.  On the top left is an element $\pi$ of $\Pi^1_{n,N}$, that is an up-right path $\pi$   from $(1,1)$ to $(n,N)=(4,5)$. The weight of this path is $d_{11}d_{21}d_{22}d_{23}d_{33}d_{43}d_{44}d_{45}$.
On the top right is an element $\pi=(\pi_1,\pi_2)$ of $\Pi^2_{n,N}$,  a pair of nonintersecting up-right paths: $\pi_1$  from $(1,1)$ to $(n,N-1)=(4,4)$ and $\pi_2$   from $(1,2)$ to $(n,N)=(4,5)$.
The weight of $\pi$ is $(d_{11} d_{21} d_{22} d_{23} d_{33} d_{43} d_{44} )\cdot\big(d_{12} d_{13} d_{14} d_{24} d_{34} d_{35} d_{45} \big)$.
At the bottom is the array
$\zarr(n)=\{\zarr_{k,\ell}(n)\}_{1\leq \ell\leq k\leq N}$, at time $n=4$.
We write this as $\zarr(n)=P_{n,N}(d)$.
The entry $\zarr_{N,1}(n)=\zarr_{5,1}(4)$ is equal to the sum of the weights of all paths in $\Pi^1_{n,N}$.
The product $\zarr_{N,1}(n)\zarr_{N,2}(n)=\zarr_{5,1}(4)\zarr_{5,2}(4)$ is equal to the sum of the weights of all elements in $\Pi^2_{n,N}$.   The rest of the   array is determined   similarly via \eqref{ztau-3}.
We regard the boldface element $\mathbf{\zarr_{55}(4)}$ as not yet defined at time
$n=4$ as explained after \eqref{ztau-3}. }
\label{pathfig}\end{figure}

We express the mapping  \eqref{ztau-3}  that defines $\zarr(n)$ from $d^{[1,n]}$ as
\be \zarr(n)=P_{n,N}(d^{[1,n]}).\label{PnN}\ee
We come to the  important point from
  Section 2.2 of \cite{NY}  that   row insertion into an empty array and
this   path construction define the same array $\zarr(n)$.  We postpone the proof
of this proposition to the end of the section.

\begin{proposition}\label{NYrecprop}  Let $n,N\ge 1$. Set
  $\zarr(n) = P_{n,N}(d^{[1,n]})$ and
\begin{equation}
\tilde{\zarr}(n) = \es \leftarrow d^{[1]}\leftarrow d^{[2]}\leftarrow \cdots \leftarrow d^{[n]}.
\label{ztilde}\end{equation}
Then $\zarr(n)=\tilde{\zarr}(n)$.
\end{proposition}

Let us discuss similarities with the classical RSK correspondence.
 Array  $z=P_{n,N}(d^{[1,n]})$ is the analogue of the $P$-tableau in the usual RSK correspondence.  The analogue of the shape of this tableau is the
bottom vector $(\zarr_{N,1},\ldots,\zarr_{N,N\wedge n})$.  The analogue of the $Q$-tableau is the array $w=Q_{n,N}(d^{[1,n]})=P_{N,n}((d^{[1,n]})^T)$, where superscript $T$ denotes transpose.  It is not difficult to see that the pair $(z,w)$ have the same `shape', that is, $(z_{N,1},\ldots,z_{N,N\wedge n})= (w_{n,1},\ldots,w_{n,N\wedge n})$.  Given this constraint, the pair $(z,w)$ can be identified with an $n\times N$ matrix $\tilde d$ defined by
\[  \tilde d_{ij} = \begin{cases} \zarr_{i+j-1,i},  &1\le i\le n\wedge N, \, 1\le j\le N-i+1\\
 w_{N+n-i-j+1, N-j+1}, &1\le j\le N,\, N-j+1\le i\le n.    \end{cases}   \]
 With this identification, the mapping $d^{[1,n]}\mapsto (z,w)$ is a bijection from the set of $n\times N$ matrices with strictly positive real entries onto itself, which one can refer to as the geometric RSK correspondence.

Proposition \ref{NYrecprop} shows that, as with the usual RSK correspondence,
the `$P$-tableau' $z(n)=P_{n,N}(d^{[1,n]})$ can be defined recursively by inserting the rows of the matrix $d$ one after another.  Then   $w=Q_{n,N}(d)$ plays the role of a `recording' tableau.  With $z(k)=P_{k,N}(d^{[1,k]})$ for $1\le k\le n$,  $w=Q_{n,N}(d)$ is given by $w_{k,\cdot}=z_{N,\cdot}(k)$ for $1\le k\le N$: this is immediate from the definition.



Let us introduce some  further conventions for the sequel. As before
$\zarr_{\ell} = \big(\zarr_{\ell,\ell},\dotsc,\zarr_{N,\ell}\big)$ is  the $\ell^{\text{th}}$ diagonal of $\zarr$, counting from right to left.
The $k^{\text{th}}$ row of the array is denoted by $\zarr^{[k]}=(z_{k1},\dotsc,z_{kk})$,
and an array restricted to a range of rows is denoted  by
 $\zarr^{[a,b]}=(\zarr_{k,\ell}: a\le k\le b,\, 1\le\ell\le   k)$.
To discuss evolution of the array row by row it is convenient to have notation
for   spaces of   rows.  For $1\le k\le N$
  the  $k^{\text{th}}$ row of an array in $\TN$  lies in the space $\SN{k}=(y_{\ell}: 1\le \ell\le k,\textrm{and } y_\ell\in (0,\infty))$, i.e., the space of vectors of length $k$ with positive real coordinates.

\medskip

As the last item of this section we sketch the proof of Proposition \ref{NYrecprop} from \cite{NY}.
\begin{proof}[Proof of Proposition \ref{NYrecprop}]
The connection between $\zarr(n)$ and $\tilde\zarr(n)$ goes via  the variables
$\tau_{k\ell}(n)$ and  a matrix formalism developed in \cite{NY}.

For an $N$-vector $x=(x_1,\dotsc,x_N)$ define an upper triangular  $N\times N$ matrix
\[ H(x)=\sum_{1\le i\le j\le N} x_ix_{i+1}\dotsm x_j E_{i,j}\] where $E_{i,j}$ is
the $N\times N$ matrix with a unique $1$ in the $(i,j)$-position and zeroes elsewhere.
For a fixed $n$, define the product
\be H=H(d^{[1]})H(d^{[2]})\dotsm H(d^{[n]}).  \label{Hprod}\ee
A key fact \cite[Prop.~1.3]{NY} is that the $\tau_{k\ell}(n)$'s give certain minor determinants
of $H$:   $\tau_{k\ell}(n)=\det H^{[1,\ell]}_{[k-\ell+1,k]}$ where the   superscript
specifies the range of rows and the subscript the range of columns in the minor.

On the other hand, the row insertion procedure can be encoded with $H$-type
matrices.   Let $1\le m\le N$ and introduce this further definition for an
$(N-m+1)$-vector $x=(x_m,x_{m+1},\dotsc, x_N)$:
\be H_m(x)= \sum_{1\le i<m} E_{i,i} + \sum_{m\le i\le j\le N} x_ix_{i+1}\dotsm x_j E_{i,j}.
\label{Hm}\ee
In particular, $H_1(x)=H(x)$.

  Set $M=n\wedge N$.   With $H$ as in \eqref{Hprod} above,   consider the equation
\be H =  H_M(\zratio_{M})\dotsm H_2(\zratio_{2})H_1(\zratio_1)  \label{H2}\ee
for unknown vectors $\zratio_{\ell}=(\zratio_{k,\ell})_{\ell\le k\le N}$, $\ell=1,2,\dotsc,M$.
This equation is uniquely solved by \cite[Thm.~2.4]{NY}
\be  \zratio_{\ell,\ell}=\frac{\tau_{\ell,\ell}(n)}{\tau_{\ell,\ell-1}(n)}\,, \quad
\zratio_{k,\ell}=\frac{\tau_{k,\ell}(n) \,\tau_{k-1,\ell-1}(n)}{\tau_{k,\ell-1}(n) \, \tau_{k-1,\ell}(n)}
 \quad\text{ for $\ell<k\le N$.}  \label{H4}\ee
 Equation \eqref{H2} encodes the row insertion procedure but in different variables
 \cite[eqn.~(2.38)--(2.40)]{NY}.  Namely,
 the $\zratio$-variables are the ratios of the $\tilde\zarr$-variables defined by \eqref{ztilde}:
 \be  \text{for $1\le\ell\le M$:} \quad  \zratio_{\ell,\ell}=\tilde\zarr_{\ell,\ell}(n)\, ,  \quad
 \zratio_{k,\ell}=\frac{\tilde\zarr_{k,\ell}(n)}{\tilde\zarr_{k-1,\ell}(n)}
 \quad \text{ for $\ell<k\le N$.}  \label{H6}\ee

Combining \eqref{H4}--\eqref{H6} for $\tilde\zarr$ with \eqref{ztau-3} for $\zarr$ gives
\be  \tilde\zarr_{k,\ell}(n) = \zratio_{\ell,\ell}\dotsm \zratio_{k,\ell}
= \frac{\tau_{k,\ell}(n)}{\tau_{k,\ell-1}(n)}  = \zarr_{k,\ell}(n)
\quad\text{for  $1\le\ell\le M$, $\ell\le k\le N$.}  \qedhere
 \label{H8}\ee
\end{proof}

\section{Geometric RSK with random input}\label{raninputsec}

Given an initial (possibly random) state $\zarr(0)\in \TN$ and a weight matrix $d$ composed of independent random rows $d^{[i]}$, with $\zarr(0)$ and $d$ independent,
 Proposition \ref{NYrecprop} shows that $\zarr(n) = P_{n,N}(d^{[1,n]})$ has the structure of a Markov process with time parameter $n$. The exact form of the transition kernel depends on the distribution of the rows $d^{[i]}$ and can be explicitly written down by appealing to the recursion of Definition \ref{NYdef}. We do this  for a particular  solvable distribution on the elements $d_{ij}$ that we now introduce.

\begin{definition} Let $\theta$ be a positive real.
A random variable $X$ has {\it inverse-gamma distribution with parameter $\theta>0$}  if it is supported on the positive reals
where it has distribution
\begin{equation}\label{invgammadensity}
\PP(X\in dx) = \frac{1}{\Gamma(\theta)} x^{-\theta-1}\exp\left\{-\frac{1}{x}\right\} dx.
\end{equation}
We abbreviate this   $X\sim \Gamma^{-1}(\theta)$.
\end{definition}

\begin{definition}
An {\it inverse-gamma weight matrix}, with respect to a {\it parameter matrix} $\thetarc=(\thetarc_{i,j}>0:  i\ge 1, 1\le j\le N)$, is a matrix of positive weights $(d_{i,j}:  i\ge 1, 1\le j\le N)$ such that the entries are  independent random variables and  $d_{i,j}\sim \Gamma^{-1}(\thetarc_{i,j})$. We call a parameter matrix $\thetarc$ {\it solvable} if $\thetarc_{i,j} =  \thetar_i+ \thetac_j>0$ for real parameters $(\thetar_i: i\geq 1)$ and $(\thetac_j: 1\le j\le N)$.   In this case we also  refer to the associated weight matrix as solvable. Column $n$ of the parameter matrix $\thetarc$ is denoted by $\thetarc^{[n]}=(\thetarc_{n,j})_{1\le j\le N}$. We denote the vector $\thetacv=(\thetac_j: 1\le j\le N)$ for later use.
\label{def-d}\end{definition}

The intertwining properties we discuss in Section \ref{intertwining} are the reason we
restrict to inverse-gamma distributed weights and the reason for the particular form
  $\thetarc_{i,j} =  \thetar_i+ \thetac_j $.

The transition kernel for the Markov chain $\zarr(n)$ on the state space $\TN$
 resulting from applying the geometric RSK correspondence to a solvable inverse-gamma weight matrix is denoted by $\Piker{N}{\thetarc^{[n]}}(\zarr,d\tilde \zarr)$.
 This represents the time $n$ transition  $\zarr(n-1) \to  \zarr(n)$.

To explicitly state this kernel, it turns out useful to exploit another structural property of the image of the geometric RSK with   independent weights $d_{n,k}$:
 the rows of the   array $\zarr(n)$ form a Markov chain (indexed from top to bottom) with respect to adding columns to the weight matrix $d$. For this purpose let us denote $z=z(n-1)$ and $\tilde{z}=z(n)$. We begin at the top of the array. The singleton top row (denoted by $y$) of $\zarr$ is updated at time $n$ by the transition kernel
\begin{equation}
\Pker{1}{\thetarc^{[n]}}(y,d\tilde y) = \Gamma(\thetarc_{n,1})^{-1} \biggl(\frac{y_1}{\tilde y_1}\biggr)^{\thetarc_{n,1}} \exp\biggl\{-\frac{y_1}{\tilde y_1}\biggr\} \frac{d\tilde y_1}{\tilde y_1}.
\end{equation}
This simply encodes $\tilde y=d_{n,1}y$ with $d_{n,1}\sim \Gamma^{-1}(\thetarc_{n,1})$.

Now we move down along the rows of the array (recall $\zarr^{[k-1]}=(\zarr_{k-1,\ell})_{1\le\ell\le k-1}$ and likewise for $\tilde\zarr^{[k-1]}$).
Given both the initial and updated row $k-1$, $(\zarr^{[k-1]}, \tilde \zarr^{[k-1]})$, from \eqref{NYalg1} we read off the time $n$ rule for updating row $k$ from
$\zarr^{[k]}=(\zarr_{k,\ell})_{1\le\ell\le k}$ to $\tilde \zarr^{[k]}=(\tilde \zarr_{k,\ell})_{1\le\ell\le k}$. The new input weight is $a_{k,1}=d_{n,k} \; \sim \,\Gamma^{-1}(\thetarc_{n,k})$,
 and the equations are
\be\begin{aligned}
\tilde\zarr_{k,1}&=a_{k,1}(\zarr_{k,1}+\tilde \zarr_{k-1,1}) \\[6pt]
\tilde\zarr_{k,\ell} &= \ddd\frac{\zarr_{k, \ell-1}\tilde \zarr_{k-1,\ell-1}}{\zarr_{k-1,\ell-1}}\cdot
   \frac{\zarr_{k, \ell}+\tilde \zarr_{k-1,\ell}}{\zarr_{k, \ell-1}+\tilde \zarr_{k-1,\ell-1}} \qquad\text{for } 2\le\ell\le k-1,\\[9pt]
\tilde\zarr_{k,k} &= \ddd\frac{\zarr_{k,k}\zarr_{k,k-1}\tilde \zarr_{k-1,k-1}}
   {(\zarr_{k, k-1}+\tilde \zarr_{k-1,k-1})\zarr_{k-1,k-1}} .
\end{aligned}\label{row-k-eq}\ee

Equations \eqref{row-k-eq} above  show that $\tilde \zarr^{[k]}$ is an explicit function of
$\zarr^{[k-1]},\zarr^{[k]},\tilde \zarr^{[k-1]}$ and the (random) weight $a_{k,1}$. Taking the push-forward of the measure on $a_{k,1}$ under this function yields a probability distribution on $\tilde \zarr^{[k]}$, as a function  of
$(\zarr^{[k-1]},\zarr^{[k]},\tilde \zarr^{[k-1]})$.
We denote this probability distribution by the stochastic kernel $\Lker{k}{\thetarc_{n,k}}\big((\zarr^{[k-1]},\zarr^{[k]};\tilde \zarr^{[k-1]}), d\tilde \zarr^{[k]}\big)$ from
 $\SN{k-1}\times \SN{k}\times \SN{k-1}$ into $\SN{k}$. Explicitly it is given as follows in terms
of its integral against a bounded Borel test function $g$ on $\SN{k}$:
\be\begin{aligned}
&\int_{\Rplus^k} g(\tilde y)\,\Lker{k}{\thetarc_{n,k}}\big((x,y;\tilde x), d\tilde y\big)
=  \int_{\Rplus} \frac{d\tilde y_1}{\tilde y_1}\, {\Gamma(\thetarc_{n,k})^{-1}}  \left(\frac{y_1+\tilde x_1}{\tilde y_1}\right)^{\thetarc_{n,k}}\exp\left\{-\frac{y_1+\tilde x_1}{\tilde y_1}\right\} \, \\[6pt]
&\qquad\qquad \times g\biggl( \tilde y_1\,,\,
\biggl\{ \frac{y_{ \ell-1}\tilde x_{\ell-1}}{x_{\ell-1}}\cdot
   \frac{y_{ \ell}+\tilde x_{\ell}}{y_{ \ell-1}+\tilde x_{\ell-1}} \biggr\}_{2\le\ell\le k-1} \,,\,
 \frac{y_{k}y_{k-1}\tilde x_{k-1}}{x_{k-1}(y_{k-1}+\tilde x_{k-1})}     \biggr) .
\end{aligned}\label{Lker}\ee

 Now we can write down the kernel for the evolution of the array. The kernel $\Piker{N}{\thetarc^{[n]}}(\zarr,d\tilde \zarr)$ for the transition from
 $\zarr=\zarr(n-1)$ to $\tilde\zarr=z(n)$  on the space $\TN$ is defined inductively on $N$.
For $N=1$ set
$\Piker{1}{\thetarc^{[n]}}= \Pker{1}{\thetarc^{[n]}}$, and for $N\ge 2$
\be
\Piker{N}{\thetarc^{[n]}}\big(\zarr^{[1,N]},d\tilde \zarr^{[1,N]}\big)
= \Piker{N-1}{\thetarc^{[n]}}\big(\zarr^{[1,N-1]},d\tilde \zarr^{[1,N-1]}\big)\,
\Lker{N}{\thetarc_{n,N}}\big((\zarr^{[N-1]},\zarr^{[N]}; \tilde \zarr^{[N-1]}), d\tilde \zarr^{[N]}\big).
\label{Piker}\ee

\subsection{Intertwining relation}\label{intertwining}
Let us first recall a well-known criterion for a function of a Markov chain to retain
the Markov property (see, for example, \cite{PR}).
Consider a measurable transformation $\phi:T\to S$ where $(T,\mathcal{T})$ and $(S,\mathcal{S})$ are   measurable  spaces. Given  Markov transition kernels $\Pi_n$ on $T$ one forms a Markov process $\zarr(n)$ which has a given initial distribution $\zarr(0)$ and transitions between $\zarr(n-1)$ to $\zarr(n)$ via $\Pi_n$. The process $\{\zarr(n)\}_{n\geq 0}$ is Markovian with respect to its own filtration. The question is, under what conditions is $y(n) = \phi(\zarr(n))$ Markovian {\it under its own filtration} $\sigma\{y(0),\ldots,y(n)\}$, and what is the associated transition kernel $\bar{P}_n:S\to S\,$?

In order to answer this question we  introduce an  {\it intertwining} kernel    $\bar{K}:S\to T$.  (The reason $\bar P$ and $\bar K$ have bars is that one often initially deals with unnormalized kernels and then normalizes them to be probability measures in their second variable.)
\begin{proposition}\label{MarkovfunctionsProp}
Assume that there exist $\bar{P}_n$ and $\bar{K}$ (which does not depend on $n$) satisfying
\begin{enumerate}
\item[(1)] for all $y\in S$, $\bar{K}(y,\phi^{-1}(y)) = 1$,
\item[(2)] for all $n$, $\bar{K} \Pi_n = \bar{P}_n \bar K$.
\end{enumerate}
Then, for any initial (possibly random) state $y^0\in S$, if one initializes the Markov chain $z(n)$ with $z(0)$ distributed according to the measure $\bar{K}(y^0,\cdot)$, one has these properties:
\begin{enumerate}
\item[(i)] For all $y\in S$ and all bounded Borel functions $f$ on $T$, $$\EE\big[f(z(n))|y(0),\ldots, y(n-1),y(n)=y\big] = (\bar Kf)(y).$$
\item[(ii)] The process $y(n)=\phi(z(n))$ is Markov in its own filtration $\sigma\{y(0),\ldots,y(n)\}$ with
 transition kernel $\bar{P}_n$.
\end{enumerate}
\end{proposition}

Return  to geometric RSK with the solvable inverse-gamma weight matrix. Take $T$ above to be $\TN$ and $S$ to be $\SN{N}$. Let  $\phi:\TN\to \SN{N}$   project  $\zarr\in \TN$ onto its bottom row. It is easier to first introduce unnormalized kernels and prove  intertwining, and then to normalize them to apply the above results.

Define a time $n$ positive kernel on $\SN{N}$ by
\begin{equation}\label{PNdef}
\Pker{N}{\thetarc^{[n]}}(y,d\tilde y) = \prod_{i=1}^{N-1} \exp\left\{-\frac{\tilde y_{i+1}}{y_i}\right\} \prod_{j=1}^{N} \biggl( \Gamma(\thetarc_{n,j})^{-1} \biggl(\frac{y_j}{\tilde y_j}\biggr)^{\thetarc_{n,j}} \exp\biggl\{-\frac{y_j}{\tilde y_j}\biggr\} \frac{d\tilde y_j}{\tilde y_j}\,\biggr).
\end{equation}
Define a positive intertwining kernel from $\SN{N}$ to $\TN$ by
\begin{equation}
\Kker{N}{\thetacv}(y,dz)=  \prod_{1\le \ell\le k<N}
\biggl(\frac{\zarr_{k,\ell}}{\zarr_{k+1, \ell}} \biggr)^{\thetac_{k+1}-\thetac_\ell}
\exp\left(
-\,\frac{\zarr_{k,\ell}}{\zarr_{k+1, \ell}}-\frac{\zarr_{k+1,\ell+1}}{\zarr_{k,\ell}} \right) \frac{d\zarr_{k,\ell}}{\zarr_{k,\ell}} \, \prod_{\ell=1}^N \delta_{y_\ell}(d\zarr_{N\ell})
\label{Kker}\end{equation}
where $\delta_{y}(d\zarr_{ij})$ is the Dirac delta measure at $y$. $\Kker{1}{\thetac_{1}}$ is the identity kernel.
Observe that $\Kker{N}{\thetacv}$ only depends on the column parameters
  $\thetacv$ which do not change with the time index $n$.
Definition \ref{def-d}  stipulated the form  $\thetarc_{i,j}=\thetar_i+\thetac_j$ for  the solvable parameter matrix in order
  to make the intertwining work.  The intertwining itself does also work with a time-dependent
$K$-kernel. But in our case computations reveal that   application of  Proposition \ref{MarkovfunctionsProp} requires a time-independent $K$  and   we are not permitted any more general $(\thetarc_{i,j})$.

$\Pker{1}{\thetarc^{[n]}}$  is a stochastic kernel (i.e., normalized to have measure one)  and represents the update
$\tilde \zarr_{1,1}=a_{1,1}\zarr_{1,1}$. But for $N\ge 2$, $\Pker{N}{\thetarc^{[n]}}$ is substochastic. This is evident because
 the second product is the stochastic kernel  of independent inverse-gamma distributed multiplicative jumps, while the first product is   a killing potential.  A Doob $h$-transform
 (or ground-state transform)  will suffice to renormalize this kernel as well as the intertwining kernel.

The main algebraic content of the integrability or solvability of geometric RSK is:
\begin{proposition}\label{itpr1}
The following intertwining relation holds at all times $n\ge 1$:
\begin{equation}\label{itPKKPI}
\Pker{N}{\thetarc^{[n]}}\Kker{N}{\thetacv} = \Kker{N}{\thetacv}  \Piker{N}{\thetarc^{[n]}},
\end{equation}
where both sides are  operators from $\SN{N}$ to $\TN$.
\end{proposition}
This is proved in Section \ref{itpr1proofsec}.

\begin{rem}
The above proposition along with the existence of the kernels $\Pker{N}{\thetarc^{[n]}}$ and $\Kker{N}{\thetacv}$ are the key to the main argument of this paper.  A priori, the existence of such kernels is not guaranteed.  The inspiration behind finding these kernels came from the related work \cite{OCon} (in which the kernel $\Kker{N}{\thetacv}$ also appears and plays a similarly important role) as well as earlier work in which analogous results were obtained in the context of the usual RSK correspondence~\cite{OCon2,OCon3,BBO1,BBO2} (see also \S\S~\ref{pitman} below). In the context of the RSK correspondence with geometric or exponentially distributed weights, the kernel $\Kker{N}{\thetacv}(y,dz)$ is replaced by the indicator function over Gelfand-Tsetlin patterns (interlacing triangular arrays) $z$ with bottom row $y$, with an exponential factor depending on the type of the pattern; and the kernel
$\Pker{N}{\thetarc^{[n]}}(y,d\tilde y)$ is replaced by the substochastic kernel of independent geometrically or exponentially distributed additive jumps from
$y$ to $\tilde y$ subject to the constraint that $\tilde y$ is interlaced with $y$.
\end{rem}

\medskip

The kernels above are not normalized. However, using the intertwining relation it is now simple to determine the necessary normalizing functions. For $y\in \SN{N}$, define
\begin{equation}
\wfunc{N}{\thetacv}(y)=\int_{\TN} \Kker{N}{\thetacv}(y,dz).
\label{wfunct}
\end{equation}
  Integrating  the intertwining (\ref{itPKKPI})  yields the eigenfunction relation
\be \Pker{N}{\thetarc^{[n]}} \wfunc{N}{\thetacv}=\wfunc{N}{\thetacv}.
\label{Pw=w}\ee
Thus we can define a stochastic kernel on $\SN{N}$  by
\begin{equation}
\Pbarker{N}{\thetarc^{[n]}}(y,d\tilde y)=
\frac{\wfunc{N}{\thetacv}(\tilde y)}{\wfunc{N}{\thetacv}(y)}
\Pker{N}{\thetarc^{[n]}}(y,d\tilde y)
\label{Pbarker}\end{equation}
and from $\SN{N}$ to $\TN$ by
\be
\Kbarker{N}{\thetacv}(y,dz)= \frac1{\wfunc{N}{\thetacv}(y)} \Kker{N}{\thetacv}(y,dz).
\label{Kbarker}\ee
The kernel $\Kbarker{N}{\thetacv}(y,dz)$ should be interpreted as the distribution of
the pattern in $\TN$ conditioned on the bottom row $\zarr^{[N]}$ being equal to $y$.
From the previous proposition follows:

\begin{corollary}\label{itpr2}
The following intertwining relation holds at all times $n\ge 1$:
\begin{equation}\label{itQKKPI}
\Pbarker{N}{\thetarc^{[n]}}\Kbarker{N}{\thetacv} = \Kbarker{N}{\thetacv}  \Piker{N}{\thetarc^{[n]}},
\end{equation}
where both sides are  operators from $\SN{N}$ to $\TN$.
 \end{corollary}

\subsection{Whittaker functions}\label{WhitFUnSect} For the next stage, note that the kernels above
remain perfectly well-defined if we allow   parameter vector $\thetacv$ to be complex.
The probabilistic meanings are lost but the intertwining continues to work.
For $y\in (0,\infty)^N$ and $\lambda\in \C^N$, define
\begin{equation}\label{Whittaker 1}
\Mker{N}{\lambda}(y,dz)= \prod_{i=1}^N y_i^{-\lambda_i} \Kker{N}{\lambda} (y,dz)
\end{equation}
and
\begin{equation}\label{Whittaker 2}
\wf{N}{\lambda}(y)=\int_{\TN} \Mker{N}{\lambda}(y,dz) = \prod_{i=1}^N y_i^{-\lambda_i}\wfunc{N}{\lambda}(y).
\end{equation}
The functions $\wf{N}{\lambda}$, well-defined for any $\lambda\in\C^N$, are class-one $GL(N,\R)$-Whittaker functions (in multiplicative variables).
They arise in various contexts:  they are eigenfunctions of the quantum Toda lattice (when expressed in additive variables $x_i=\log y_i$)
and can be represented as particular matrix elements of infinite-dimensional representations of ${\mathfrak gl}(N)$~\cite{k};
they also arise in the harmonic analysis of automorphic
forms on $GL(N,\R)$~\cite{bump}. The integral representation (\ref{Whittaker 2}) is due to Givental~\cite{givental}.
It is known~\cite{sts,KL} that the integral transform
\begin{equation}\label{WhittakerTransform}
\hat{f}(\lambda) = \int_{\RNplus} \prod_{i=1}^{N} \frac{dy_i}{y_i} f(y) \wf{N}{\lambda}(y)
\end{equation}
defines an isometry of $L_2(\RNplus,\prod_i\,dy_i/y_i)$ onto $L_2^{sym}(\iota \R^N,s_N(\lambda)d\lambda)$,
where $L_2^{sym}$ is the space of $L^2$ functions which are symmetric in their variables, $\iota=\sqrt{-1}$ is the imaginary unit and
\begin{equation}\label{Whittaker 7}
s_N(\lambda)=\frac1{(2\pi \iota)^N N!} \prod_{j\ne k} \Gamma(\lambda_j-\lambda_k)^{-1}.
\end{equation}
The inversion formula is
\begin{equation}\label{InverseTransform}
\check{g}(y) = \int_{\iota \R^N} g(\lambda) \overline{\wf{N}{\lambda}(y)} s_N(\lambda)  d\lambda.
\end{equation}
In particular, the Plancherel formula
\begin{eqnarray}\label{Planch}
\int_{\RNplus}f(y)\overline{g(y)} \prod_{i=1}^N\frac{dy_i}{y_i}=\int_{\iota\mathbb{R}^N} \hat{f}(\lambda)\overline{\hat{g}(\lambda)}s_N(\lambda)d\lambda
\end{eqnarray}
holds for functions $f,g\in L^2(\RNplus,\prod\,dy_i/y_i)$.





We also have the Whittaker integral identity~\cite{B,St,GLO}, for $s>0$
and $\lambda,\nu\in\C^N$, 
\begin{equation}\label{bs}
\int _{\RNplus}e^{-s y_1} \wf{N}{\lambda}(y)\wf{N}{\nu}(y)\prod_{i=1}^N\frac{dy_i}{y_i} =
s^{\sum (\lambda_i+ \nu_i)} \prod_{i,j}\Gamma(-\lambda_i-\nu_j).
\end{equation}
Using $\Psi_\theta(y)=\Psi_{-\theta}(y')$, where $y'_i=y_{N-i+1}^{-1}$, this is equivalent to
\begin{equation}\label{bs'}
\int _{\RNplus}e^{-s y_N^{-1}} \wf{N}{\lambda}(y)\wf{N}{\nu}(y)\prod_{i=1}^N\frac{dy_i}{y_i} =
s^{-\sum (\lambda_i+ \nu_i)} \prod_{i,j}\Gamma(\lambda_i+\nu_j).
\end{equation}

Note that if, for $\zarr\in \TN$, we define $x_i(z)$ by
$$\prod_{i=1}^k x_i(z) = \prod_{i=1}^k \zarr_{k,i},\qquad k=1,\ldots,N,$$
then
\begin{equation}\label{Whittaker 3}
\Mker{N}{\thetacv}(y,dz)=\prod_{i=1}^N x_i(z)^{-\theta_i} \Mker{N}{}(y,dz),
\end{equation}
where
\begin{equation}\label{Whittaker 4}
\Mker{N}{}(y,dz)= \prod_{1\le \ell\le k<N}
\exp\left(
-\,\frac{\zarr_{k,\ell}}{\zarr_{k+1, \ell}}-\frac{\zarr_{k+1,\ell+1}}{\zarr_{k,\ell}} \right) \frac{d\zarr_{k,\ell}}{\zarr_{k,\ell}} \, \prod_{\ell=1}^N \delta_{y_\ell}(d\zarr_{N\ell}).
\end{equation}
For $n\ge 0$ and $i=1,\ldots,N$ write $x_i(n)=x_i(\zarr(n))$.  Then $x_i(n)$ is a
multiplicative random walk:
\begin{equation}\label{Whittaker 5}
x_i(n)=\left(\prod_{m=1}^n d_{m,i}\right) x_i(0).
\end{equation}

\subsection{Main theorems}
We  are prepared to state the two main theorems of the paper. They are proved in Sections \ref{markovprojthmproofsec} and \ref{main corollaryproofsec}, respectively. The first result is concerned with the solvability of the $\proj$ projection of the $\zarr(n)$ Markov chain corresponding to the recursive system in equation (\ref{NYalg1}).

\begin{theorem}\label{markovprojthm}
Fix a solvable inverse-gamma weight matrix defined in terms of parameters $(\thetar_m: m\geq 1)$ and $(\thetac_j: 1\le j\le N)$ and assume (without loss of generality) that   $\thetac_j<0<\thetar_m$ for all $j,m$.
Let $y(0)$ be a random or deterministic initial state in $\SN{N}$ and let the initial distribution of $\zarr(0)$ be
$\Kbarker{N}{\thetacv}(y(0),\,\cdot\,)$.
\begin{enumerate}
\item[(i)]
The sequence of random variables $y(n)=\proj(\zarr(n)), n\ge 0$, is a Markov chain with respect to its own filtration,
with state space $\SN{N}$, initial state $y(0)$  and time $n$ transition kernel $\Pbarker{N}{\thetarc^{[n]}}$.
\item[(ii)] For a bounded Borel function $f$ on $\TN$ and $y\in\SN{N}$
\begin{eqnarray}\label{Thm rel 1}
 E[ f(\zarr(n))\,\vert\, y(0),\dotsc,y(n-1), y(n)=y]
= \int_{\TN} \Kbarker{N}{\thetacv}(y,dz)\, f(z).
\end{eqnarray}
\item[(iii)]  For $\lambda\in\C^N$
\begin{eqnarray}\label{Thm rel 2}
E\Bigl[\;\prod_{i=1}^N x_i(n)^{-\lambda_i}\Big\vert \, y(0),\dotsc,y(n-1), y(n)=y\Bigr]
 = \frac{\wf{N}{\thetacv+\lambda}(y)}
{\wf{N}{\thetacv}(y)}.
\end{eqnarray}
\item[(iv)]
For an initial state $y^0\in\SN{N}$ and time $\ge 1$,  let $\mu_n^N(y^0,dy)$
denote the probability distribution of the time-$n$ state $y(n)$.
 Then for all $\lambda\in\iota\R^N$,
\begin{eqnarray}\label{Thm rel 3}
 \int_{\RNplus} \frac{\wf{N}{\lambda}(y)}{\wf{N}{\thetacv}(y)} \mu_n^N(y^0,dy)
= \frac{\wf{N}{\lambda}(y^0)}{\wf{N}{\thetacv}(y^0)} \prod_{m=1}^n\prod_{i=1}^N
\frac{\Gamma(\thetar_m+\lambda_i)}{\Gamma(\thetac_i+\thetar_m)}.
\end{eqnarray}
Moreover, for any continuous, compactly supported function $f$ on $\RNplus$ we have
\begin{eqnarray}\label{int_mu}
&&\\
&&\int_{\mathbb{R}^N_+} f(y) \mu_n^N(y^0,dy)=
\int_{\iota \mathbb{R}^N} d\lambda \, s_N(\lambda)  \frac{\Psi^N_{\lambda}(y^0)}{\Psi^N_\theta(y^0)}
 \left( \int_{\mathbb{R}^N_+} f(y) \Psi^N_\theta(y) \Psi^N_{-\lambda}(y) \prod_i\frac{dy_i}{y_i}\right)
 \prod_{m=1}^n\prod_{i=1}^N\frac{\Gamma(\hat\theta_m+\lambda_i)}{\Gamma(\theta_i+\hat\theta_m)}
 \nonumber.
\end{eqnarray}


\end{enumerate}
\end{theorem}

\begin{rem}\label{spec-remark}
Note that, given part (i) of the above theorem,
\eqref{Thm rel 3} shows that one may diagonalize the transition kernel
via the eigenfunction equation
\be\label{Pspec}
  \Pbarker{N}{\thetarc^{[n]}} \frac{\wf{N}{\lambda}}{\wf{N}{\thetacv}}
  =
 \biggl(\; \prod_{j=1}^N
\frac{\Gamma(\thetar_n+\lambda_j)}{\Gamma(\thetac_j+\thetar_n)} \biggr)
\frac{\wf{N}{\lambda}}{\wf{N}{\thetacv}},
\ee
which can also be seen directly from the intertwining relation (\ref{itPKKPI}), cf. (\ref{Pw=w}). By applying the completeness relation resulting from the $L^2$ isometry, this identity characterizes the transition kernel.
\end{rem}

Next we specialize the above result to the Markov chain $y(n)$ that comes from the evolving shape of the geometric RSK array $P_{n,N}(d^{[1,n]})$ of \eqref{PnN}. This  is the case of the  empty initial array. We can capture this situation by taking a somewhat delicate limit of the initial state $y^0$.   This is our second main result.

\begin{theorem}\label{main corollary}
Fix a solvable inverse-gamma weight matrix defined in terms of parameters $(\thetar_m: m\geq 1)$ and $(\thetac_j: 1\le j\le N)$ and assume (without loss of generality) that
$\thetac_j<0<\thetar_m$ for all $j,m$.

 Consider the array $\rho=(\rho_{k,\ell})_{1\leq \ell\leq k \leq N}$ with
\be
\rho_k=(\rho_{k,\ell})_{1\leq \ell \leq k}=\left( \frac{k-1}{2},\frac{k-1}{2}-1,\dots, -\frac{k-1}{2}\right),
\label{rho}\ee
for $1\leq k\leq N$.
Let $y^{0,M}=\left(e^{-M\rho_{N,\ell}}\right)_{1\leq \ell \leq N}$ and $n\geq N$. Then
\begin{enumerate}

\item[(i)]
As $M\to \infty$, the probability distribution $\mu_n^N(y^{0,M},dy)$ converges weakly to a distribution $\mu_n^N(dy)$ characterized by
\begin{eqnarray*}
\int_{\mathbb{R}^N_+} f(y) \mu_n^N(dy)=
\int_{\iota \mathbb{R}^N} d\lambda \, s_N(\lambda)
 \left( \int_{\mathbb{R}^N_+} f(y) \Psi_\theta(y) \Psi_{-\lambda}(y) \prod_i\frac{dy_i}{y_i}\right)
 \prod_{m=1}^n\prod_{i=1}^N\frac{\Gamma(\hat\theta_m+\lambda_i)}{\Gamma(\theta_i+\hat\theta_m)},
\end{eqnarray*}
for any continuous, compactly supported function $f$ on $\RNplus$.

\item[(ii)] The Laplace transform of the projection of $\mu_n^N(dy)$ on the first coordinate
is given by
\begin{equation}\label{laplaceeqn}
\int_{\RNplus}e^{-sy_1}\mu_n^N(dy)=
  \int_{\iota \mathbb{R}^N}d\lambda\, s^{\sum_{i=1}^N(\theta_i-\lambda_i)}
          \prod_{1\leq i,j\leq N}\Gamma(\lambda_i-\theta_j) \,\prod_{m=1}^n\prod_{i=1}^N\frac{\Gamma(\lambda_i+\hat\theta_m)}{\Gamma(\theta_i+\hat\theta_m)} s_N(\lambda),
\end{equation}
where the poles of the functions $\Gamma(\lambda_i-\theta_j)$ and $\Gamma(\lambda_i+\hat\theta_m)$ are not encountered due to the assumed condition that $\thetar_m>0$ for all $m$ and $\thetac_j<0$ for all $j$.
\item[(iii)]  The measure $\mu_n^N(dy)$ is the distribution of the bottom row $y(n)$,
given that the process begins with the empty array.  In particular,
the distribution of the partition function
$\zarr_{N,1}(n)=\sum_{\pi\in \Pi^1_{n,N}}wt(\pi)$ is the marginal distribution of $\mu_n^N(dy)$ on the first coordinate $y_1$, and hence uniquely characterized by \eqref{laplaceeqn}.
\item[(iv)] The distribution of $P_{n,N}(d^{[1,n]})$ is the measure in $dz$ given by
$$\int \mu_n^N(dy) \Kbarker{N}{\thetacv}(y,dz).$$
\item[(v)] When $n=N$ we have the following simplification:
$$\mu_N^N(dy)=\prod_{m=1}^N\prod_{i=1}^N\Gamma(\theta_i+\hat\theta_m)^{-1} e^{-y_N^{-1}} \Psi_\theta(y) \Psi_{\hat\theta}(y)\prod_{i=1}^N\frac{dy_i}{y_i}.$$
In particular, for $s>0$,
$$\int_{\RNplus}e^{-sy_N^{-1}}\mu_N^N(dy)= (1+s)^{-\sum_{i=1}^N(\theta_i+\hat\theta_i)},$$
that is, the random variable $\zarr_{N,N}(N)$ is inverse gamma distributed with parameter $\sum_{i=1}^N(\theta_i+\hat\theta_i)$.
\end{enumerate}
\end{theorem}

The distribution of $\zarr_{N,N}(N)$ can also be seen from \eqref{ztau-3}.

Observe that the condition of $n\geq N$ is not restrictive when it comes to computing the Laplace transform in part (ii) of the above theorem. Indeed, if one wishes to compute the Laplace transform for $n<N$ then it suffices to transpose the parameter matrix and switch the role of $n$ and $N$. The distribution of the coordinate $y_1$ is unchanged by this procedure, and now the above corollary applies.

\subsection{Pitman's $2M-X$ theorem}\label{pitman}
Theorem \ref{markovprojthm} can be regarded as a variant of Pitman's `$2M-X$ theorem', which states that, if $X_t$ is a standard one-dimensional Brownian motion and $M_t=\max_{s\le t} X_s$, then $2M_t-X_t$ is a three-dimensional Bessel process.
This theorem has vast generalizations~\cite{BOCon,BBO1,BBO2,BJ,DMO,KOR,MY,OCon,OCon2,OCon3,OCon4,OConYor2}, many of which have been obtained via various analogues the `Burke property' discussed in Remark \ref{burke-rem} below.
 All of these can be regarded as variations of the statement that the stochastic evolution of the shape of the tableaux, obtained when applying variants of the RSK algorithm to random input data, has the Markov property. The first `geometric' or `positive temperature' version of this statement was discovered by Matsumoto and Yor~\cite{MY}, who showed that, for $X_t$ as above, the process $\log\int_0^t e^{2X_s-X_t} ds,\ t>0$ is a diffusion on $\R$ with infinitesimal generator given by
$$
\frac12 \frac{d^2}{dx^2}+\left(\frac{d}{dx}\log K_0(e^{-x})\right) \frac{d}{dx},
$$
where $K_0$ is the Macdonald function (with index 0).  A multi-dimensional version of this theorem of Matsumoto and Yor is given in~\cite{OCon}, which can be regarded as a particular specialization (scaling limit) of the main result in the present paper.
It is also proved via an intertwining relation and is closely related to the quantum Toda lattice.  The corresponding directed
polymer model is defined on the semi-lattice $\Z\times\R$.  Both models feature the $GL(N,\R)$-Whittaker functions in an essential way; the
relation between them is analogous to the relation between the Gaussian and Laguerre unitary ensembles in random matrix theory.



\subsection{Invariant distributions} 
\label{burkesec}
The Markov process defined by row insertion with a solvable inverse gamma parameter matrix turns out to have nice  invariant distributions. The $z$-array itself cannot have an invariant distribution:  for example, $\zarr_{1,1}(n)=d_{n,1}\dotsm d_{1,1}\zarr_{1,1}(0)$ evolves as a
multiplicative random walk. Instead,  we look at ratios  of $\zarr$-values.

Fix $N\ge 1$.
For an array $z\in\TN$ define the array $\zratio=(\zratio_{k\ell})_{1\le\ell<k\le N}$ of ratios by
\[  \zratio_{k,\ell}=\frac{\zarr_{k,\ell}}{\zarr_{k-1,\ell}}\,,\quad  1\le\ell<k\le N.  \]
The Markov process $\zarr(n)$ then defines another random process
$\zratio(n)=(\zratio_{k\ell}(n))_{1\le\ell<k\le N}$  by   $\zratio_{k,\ell}(n)= {\zarr_{k,\ell}(n)}/{\zarr_{k-1,\ell}(n)}$.  Denote again diagonals by
$\zratio_\ell(n)=(\zratio_{k\ell}(n))_{\ell<k\le N}$ for $1\le\ell<N$.
 This new process $\zratio(n)$ will also  be a Markov chain.

\begin{theorem}  Let $\zarr(n)$ evolve on the space $\TN$ according to the Markovian dynamics
governed by a solvable inverse-gamma weight matrix with parameters
$\thetarc_{i,j} = \thetar_i + \thetac_j $,
 as specified by the transition kernels in \eqref{Piker}.

{\rm (a)} The process $\zratio(n)$ is a Markov chain in its own filtration.

{\rm (b)}   Let $1\le j<N$.
Assume $\thetac_1<\thetac_2<\dotsm<\thetac_j<\min\{\thetac_{j+1},\dotsc,\thetac_N\}$.  Then the process
$(\zratio_1(n),\dotsc,\zratio_j(n))$ has an invariant distribution where the variables $\{\zratio_{k\ell}: 1\le \ell\le j,\, \ell<k\le N\}$
are independent with marginal distributions
$\zratio_{k\ell}\sim\Gamma^{-1}(\thetac_k-\thetac_\ell)$.
If the process is started with this distribution, then the following statement holds
for all times $n\ge 1$:  the variables
$\{\zratio_{k\ell}(n) : 1\le\ell\le j,\, \ell<k\le N\}\cup
\{ \zarr_{N\ell}(m)/\zarr_{N\ell}(m-1): 1\le m\le n, \, 1\le\ell\le j\}$ are independent with
marginals $\zratio_{k\ell}(n)\sim\Gamma^{-1}(\thetac_k-\thetac_\ell)$ and
$ \zarr_{N\ell}(m)/\zarr_{N\ell}(m-1)\sim\Gamma^{-1}(\thetar_m+\thetac_\ell)$.
\label{z-burkethm}\end{theorem}

\begin{rem}
Theorem \ref{z-burkethm} is  an extension of a
  result of  \cite{S} for the directed polymer with inverse-gamma weights.
This   could be called a `Burke property'
by analogy with the Burke theorem (also known as the `output theorem') of M/M/1 queues.
According to the Burke theorem,  for a
reversible queue the number of customers in the system at time $t$ is independent
of the departure process up to time $t$.
This notion has an analogy  in models with random weight matrices, and it
has been used in the past to derive exact limit shapes
\cite{S98, S99}  and fluctuation exponents \cite{BCS, BalaS, CatG}.
 In fact, it was this property  that led us to investigate the solvability of
geometric RSK for  inverse-gamma weight matrices.
The analogous Burke property was found earlier in the Brownian polymer model
\cite{OConYor} and  was used to derive fluctuation exponents for that
model in \cite{SV}.
\label{burke-rem}\end{rem}

Theorem \ref{z-burkethm} is proved via  \eqref{NYalg}
that represents a transition of the Markov process $\zarr(n)$ in terms of
a sequence of geometric row insertions.   For this purpose we reformulate the row insertion
step in terms of the ratios.  In Definition \ref{NYdef} with fixed $1\le \ell<N$ the inputs of the row insertion
were $\xi=(\xi_{\ell},\ldots,\xi_N)$ and $b=(b_\ell,\ldots, b_N)$, and the outputs
  $\xi'=(\xi'_{\ell},\ldots, \xi'_N)$ and $b'=(b'_{\ell+1},\ldots, b'_N)$.   Define now
 $\zratio_k=\xi_k/\xi_{k-1}$  and $\zratio'_k=\xi'_k/\xi'_{k-1}$
  for $\ell<k\le N$, and also  auxiliary variables  $\zeta_k=\xi'_k/\xi_k$ for $\ell\le k\le N$.
The words are $\zratio=(\zratio_{\ell+1},\dotsc, \zratio_N)$ and  $\zratio'=(\zratio'_{\ell+1},\dotsc, \zratio'_N)$.

  \begin{lemma}\label{lm-NYalt}
Fix integers $1\le \ell\le N$.  In terms of the new  variables, geometric row insertion transforms
$(\zratio, b)$ into $(\zratio',b')$ via the following equations.  Set first $\zeta_\ell=b_\ell$,
and then inductively  for $k=\ell+1,\dotsc, N$:
\be
\zratio'_k=b_k\Bigl(1+ \frac{\zratio_k}{\zeta_{k-1}}\Bigr)\,, \quad
\zeta_k=b_k\Bigl(1+ \frac{\zeta_{k-1}}{\zratio_k}\Bigr)\,, \quad
\text{and}\quad     b'_k=\Bigl(\;\frac1{\zeta_{k-1}}+ \frac1{\zratio_k}\;\Bigr)^{-1}.
\label{z-burke1}\ee
\end{lemma}

Next the row insertion step with random input.

\begin{lemma}  Fix integers $1\le \ell<N$. Let $\alpha_{\ell+1}, \dotsc,  \alpha_N$, $\beta_\ell, \dotsc,  \beta_N$ be positive reals
that satisfy $\beta_k=\beta_\ell+\alpha_k$ for $\ell<k\le N$.
  Assume that  the random variables  $\{\zratio_k : \ell<k\le N\}
  \cup \{b_k: \ell\le k\le N\}$ are
   independent
with marginal distributions $\zratio_k\sim\Gamma^{-1}(\alpha_k)$  and
$b_k\sim\Gamma^{-1}(\beta_k)$.
Then the random variables  $\{\zratio'_k,b'_k: \ell<k\le N\}\cup\{\zeta_N\}$ are also  independent
with marginal distributions
 $\zratio'_k\sim\Gamma^{-1}(\alpha_k)$, $b'_k\sim\Gamma^{-1}(\beta_k)$, and
  $\zeta_N\sim\Gamma^{-1}(\beta_\ell)$.
\label{burkelm1}\end{lemma}

\begin{proof}
 From the assumptions and by definition,   $\zeta_\ell=b_\ell\sim\Gamma^{-1}(\beta_\ell)$ and
 this variable  is
independent  of $\{\zratio_{\ell+1}, \dotsc, $  $\zratio_N, b_{\ell+1}, \dotsc,  b_N\}$.
Use equations \eqref{z-burke1} to prove, inductively  on  $m=\ell+1,\dotsc, N$,
that random variables $\{\zratio'_{\ell+1}, \dotsc,  \zratio'_m, b'_{\ell+1}, \dotsc,  b'_m, \zeta_m\}$
are independent, independent of  $\{\zratio_{m+1}, \dotsc, $  $\zratio_N, b_{m+1}, \dotsc,  b_N\}$,
and their marginal distributions are
 $\zratio'_k\sim\Gamma^{-1}(\alpha_k)$, $b'_k\sim\Gamma^{-1}(\beta_k)$ and
 $\zeta_m\sim\Gamma^{-1}(\beta_\ell)$. An induction step is achieved by applying
 \eqref{z-burke1} to the triple $(\zeta_m, \zratio_{m+1}, b_{m+1})$ to produce the
new triple $(\zeta_{m+1}, \zratio'_{m+1}, b'_{m+1})$.
  Note that the parameter of $\zeta_m$ does not
 change with $m$.
 The case  $m=N$ gives the lemma.
\end{proof}

\begin{proof}[Proof of Theorem \ref{z-burkethm}]
(a) That $\zratio(n)$ is itself a Markov process follows from the fact that  from
\eqref{z-burke1} we can build autonomous equations for this evolution.

(b) It suffices to  show that the last claim  is preserved by a step of the
evolution.
Consider the time $n$ transition   from state
$\zratio=\zratio(n-1)$ to state $\zratio'=\zratio(n)$.
 The input weights are
$a_1=(a_{11},\dotsc, a_{N1})=d^{[n]}=(d_{n,1},\dotsc, d_{n,N})$ with $d_{n,k}\sim\Gamma^{-1}(\thetarc_{n,k})$, and this also defines the
first diagonal $a_1$ of the auxiliary array in Def.~\ref{NYdef-z}.  Assume that the variables
$\{\zratio_{k\ell} : 1\le\ell\le j,\, \ell<k\le N\}\cup
\{ \zarr_{N\ell}(m)/\zarr_{N\ell}(m-1): 1\le m\le n-1, \, 1\le\ell\le j\}$ are independent with
marginals $\zratio_{k\ell}\sim\Gamma^{-1}(\thetac_k-\thetac_\ell)$ and
$ \zarr_{N\ell}(m)/\zarr_{N\ell}(m-1)\sim\Gamma^{-1}(\thetar_m+\thetac_\ell)$.
Let $\zeta_{N\ell}= \zarr_{N\ell}(n)/\zarr_{N\ell}(n-1)$  denote ratios defined along the
transition process.

We prove the following statement inductively over $\ell=1,\dotsc, j$.
\be\begin{aligned}
&\text{The  variables $\{\zratio'_1,\dotsc,\zratio'_\ell, a_{\ell+1}, \zeta_{N1},\dotsc, \zeta_{N\ell}\}$
are independent and independent of}\\
&\text{$\{\zratio_{\ell+1}, \dotsc, \zratio_j\}\cup
\{ \zarr_{Ni}(m)/\zarr_{Ni}(m-1): 1\le m\le n-1, \, 1\le i\le j\}$,  \  and } \\
&\text{they have marginals $\zratio'_{ki}\sim\Gamma^{-1}(\thetac_k-\thetac_i)$,
$a_{k,\ell+1}\sim\Gamma^{-1}(\thetar_n + \thetac_k)$ and
$ \zeta_{Ni} \sim\Gamma^{-1}(\thetar_n +\thetac_i)$. }
\end{aligned}\label{ind1}\ee
The case $\ell=j$ completes the proof.

In the first row insertion step apply Lemma \ref{burkelm1} with
$\ell=1$ and inputs $\zratio_1=(\zratio_{2,1},\dotsc, \zratio_{N,1})$ and
$b=(a_{11},\dotsc, a_{N1})$.  Now $\alpha_k=\thetac_k-\thetac_1$ and
$\beta_k=\thetarc_{n,k} = \thetar_n + \thetac_k$.  According to the lemma,
the outputs $\zratio'_1=(\zratio'_{2,1},\dotsc, \zratio'_{N,1})$, $b'=a_2=(a_{22},\dotsc, a_{N2})$
 and  $\zeta_{N1}$ are independent
and they have the correct marginal distributions:
$\zratio'_{k,1}\sim\Gamma^{-1}(\thetac_k-\thetac_1)$,
$a_{k2}\sim\Gamma^{-1}(\thetar_n + \thetac_k)$  and  $ \zeta_{N1} \sim\Gamma^{-1}(\thetar_n +\thetac_1)$.
This gives \eqref{ind1} for $\ell=1$.

In the general step, assuming  \eqref{ind1} for $\ell-1$, Lemma \ref{burkelm1} is applied
to inputs $\zratio_{\ell}$ and
$a_{\ell}$,  with $\alpha_k=\thetac_k-\thetac_{\ell}$ and
$\beta_k=\thetarc_{n,k} = \thetar_n + \thetac_k$.
The outputs  $\zratio'_{\ell}$, $a_{\ell+1}$, $\zeta_{N,\ell}$ have the right
properties
and the  validity of \eqref{ind1} is exteded to $\ell$.
 \end{proof}

\section{Degenerations to known results}\label{scalinglimits}

We detail rescalings of the inverse-gamma polymer which recover known results.

\subsection{Directed last passage percolation and the Laguerre Unitary Ensemble}\label{Lag}

Fix a solvable parameter matrix $\thetarc=(\thetarc_{i,j}>0:  i\ge 1, 1\le j\le N)$ such that $\thetarc_{i,j} =  \thetar_i+ \thetac_j$. Consider a family (indexed by $\e>0$) of solvable inverse-gamma weight matrices $d^{\e} = (d_{i,j}^{\e}: i\ge 1, 1\le j\le N)$ such that the entries are  independent random variables and  $d_{i,j}^{\e}\sim \Gamma^{-1}(\e \thetarc_{i,j})$. Keeping track of the $\e$ write $z^{\e}_{k,\ell}(n)$ as the elements of $z^{\e}(n)$ -- the image of the weight matrix $(d_{i,j}^{\e}: 1\le i\le n, 1\le j\le N)$ under the geometric RSK correspondence. Write $F^{\e}(n) = (F^{\e}_{k,\ell}(n): 1\leq \ell\leq k\leq N)$ where $F^{\e}_{k,\ell}(n) = \e \log z^{\e}_{k,\ell}(n)$.

With respect to the same solvable parameter matrix consider a weight matrix $w=(w_{i,j}:i\ge 1, 1\le j\le N)$ such that the entries are independent random variables and $w_{i,j}\sim {\rm Exp}(\thetarc_{i,j})$ (an exponential random variable with rate $\thetarc_{i,j}$, or equivalently mean $(\thetarc_{i,j})^{-1}$). The classical RSK correspondence maps the weight matrix $w$ to a pair of Young tableaux $(P,Q)$ and is defined analogously but with the $(+,\times)$ algebra replaced by $(\max,+)$. We focus on the $P$-tableaux and writing it in terms of a Gelfand-Zetlin pattern (a triangular array with interlacing). Recall the notation for non-intersecting paths $\Pi_{n,k}^{\ell}$ from Section \ref{Kirillov corres}. The weight of an $\ell$-tuple $\pi$ of paths is now
\begin{equation}
\tilde{wt}(\pi) = \sum_{r=1}^{\ell} \sum_{(i,j)\in \pi_{r}} w_{i,j}.
\end{equation}
Define an array $L(n)=\{L_{k,l}(n):\ 1\le k\le N,\ 1\le l\le k\wedge n\}$ by
\begin{equation}
L_{k,1}(n)+\cdots + L_{k,\ell}(n) = \max_{\pi\in\Pi_{n,k}^{\ell}} \tilde{wt}(\pi).
\end{equation}
$L(n)$ is the Gelfand-Zetlin pattern version of the $P$-tableau of the image of $(w_{i,j}:1\le i\le n, 1\le j\le N)$ under the RSK correspondence. As in the case of the geometric RSK correspondence, when $n<N$ it is necessary to leave some entries of $L(n)$ undefined,  or
populate them with singular values (see Remark \ref{sing-rem}).

\begin{proposition}\label{limLPP}
The $n$-indexed process $\left(F^{\e}(n)\right)_{n\geq 0}$ converges in law, as $\epsilon\to 0$, to the $n$-indexed process $\left(L(n)\right)_{n\geq 0}$.
\end{proposition}
\begin{proof}
This hinges on two observations. The first is that as $\e$ goes to zero, $\e \log d_{i,j}^{\e}$ converges in distribution to $w_{i,j}$. Hence by independence the whole array
$\{\e \log d_{i,j}^{\e}\}$ converges in distribution to $w$ and (by the continuous mapping theorem) the random vectors
\begin{equation}\label{pivec}
\left(\sum_{r=1}^{\ell} \sum_{(i,j)\in \pi_r} \e \log d_{i,j}^{\e}\right)_{\pi\in\Pi^{\ell}_{n,k}, 1\leq \ell\leq k\leq N} \Longrightarrow
\left(\sum_{r=1}^{\ell} \sum_{(i,j)\in \pi_r} w_{i,j}\right)_{\pi\in\Pi^{\ell}_{n,k}, 1\leq \ell\leq k\leq N}.
\end{equation}
The second observation is that on $\R^m$, the function $f_{\e}(x) = \e \log \sum_{i=1}^{m} e^{x_i/\e}$ converges uniformly, as $\e\to 0$, to the function $f_0(x) = \max(x_i: 1\leq i\leq m)$. The process $\left(F^{\e}(n)\right)_{n\geq 0}$ is formed by applying an array of functions of the type $f_{\e}$ to the elements of the vector on the left of (\ref{pivec}). Combining the uniform convergence of functions of this type with the convergence in distribution in (\ref{pivec}), the claimed convergence of the processes follows.
\end{proof}

It is worth noting that in the above limit we only recover the  directed last passage percolation model with exponentially distributed weights, and not the  model
with geometric distribution studied in \cite{KJ}. Since our log-gamma distributions are continuous, it does not seem possible to recover the discrete geometric distribution (but rather just their continuous exponential counterparts).

Let us now briefly recall the connections between bottom row of $L(n)$ and the eigenvalues of the Laguerre Unitary Ensemble (LUE). Consider an array $(A_{i,j}:1\leq i\leq N, j\geq 1)$ of independent complex zero-mean Gaussian distributed random variables with variance $(\thetarc_{j,i})^{-1}$. We have changed the order of $i$ and $j$ since we now set $A(n)= (A_{i,j}:1\leq i\leq N, 1\leq j\leq n)$ and treat $A(n)$ as a matrix with (row,col) notation. Set $M(n) = A(n) A(n)^*$ the $N\times N$ generalized Wishart random matrix, and define for each $n$, a vector of ordered (largest to smallest) eigenvalues of $M(n)$: $\lambda(n) = (\lambda_1(n),\ldots, \lambda_N(n))$.

When $\thetarc_{i,j}=1$ for all $i,j$, Johansson \cite{KJ} showed that for fixed $n$, $\lambda_1(n) \stackrel{(d)}{=}L_{N,1}(n)$. This was strengthened by \cite{Doumerc} to show that for the same parameters as in \cite{KJ} and for $n\geq N$ fixed, the vector $\left(\lambda_i(n)\right)_{i=1}^{N} \stackrel{(d)}{=}\left(L_{N,i}(n)\right)_{i=1}^{N}$. In \cite{Defosseux, ForresterRains} equality in law was shown for the processes $\left(\lambda_1(n)\right)_{n\geq 1}$ and $\left(L_{N,1}(n)\right)_{n\geq 1}$.

Turning to the general case for the parameters $\thetarc_{i,j}$, \cite{BP07} proved that for $n$ fixed, $\lambda_1(n) \stackrel{(d)}{=}L_{N,1}(n)$ and conjectured equality of the corresponding processes in $n$. This was then proved in \cite{DW08}, whose methods (combining Theorem 3.1 and Lemma 4.1) show equality in law of the process $\left(\lambda(n)\right)_{n\geq 1}$ and $\left(L_N(n)\right)_{n\geq 1}$ where $L_N(n)= \left(L_{N,i}(n)\right)_{i=1}^{N}$ and where only non-zero eigenvalues/non-singular entries of $L$ are considered.

Combining Proposition \ref{limLPP} with the above discussion one sees that the logarithm of the bottom row of the image of a solvable weight matrix under the geometric RSK correspondence are analogous to the eigenvalues of the LUE ensemble. This connection can be seen from the integral formulas we have derived in (i) of Theorem \ref{main corollary}. We perform point-wise asymptotics to demonstrate this connection we have proved above. Performing the change of variables to that formula given by $y_i\mapsto e^{\e^{-1} x_i}$, $\theta\mapsto \e \theta$ and $\lambda\mapsto \e \lambda$ we find that the measure for $(F^{\e}_{N,\ell}(n))_{\ell=1}^{N} = (x_1,\ldots, x_N)$ is given by

$$\prod_{i=1}^N \e^{-1} dx_i \wf{N}{\e\thetacv}(e^{\e^{-1}x})
   \int_{\iota \mathbb{R}^N} \e^{N} d{\lambda} \, \wf{N}{\e(-\lambda)}(e^{\e^{-1}x})
  \prod_{m=1}^n \prod_{i=1}^N\frac{\Gamma(\e(\hat\theta_m+\lambda_i))}{\Gamma(\e(\theta_i+\hat\theta_m))}s_N(\e\lambda).$$

We may now evaluate the $\e\to 0$ asymptotics: For the Gamma functions we employ the expansion near zero; for the measure $s_N$ we employ the Euler Gamma reflection formula; for the Whittaker functions we can write the Whittaker functions in additive variables (and perform a sign change) $\psi^{N}_{\lambda}(x)  =\wf{N}{-\lambda}(e^{x})$, and then use the fact (see for instance \cite{KL,OCon}) to see 
$$\lim_{\e\to 0} \e^{\frac{N^2-N}{2}} \psi^{N}_{\e\lambda}(\e^{-1}x) = \tilde{\psi}_{\lambda}^{N}(x)  = \tilde{\psi}^{N}_{\lambda_1,\ldots, \lambda_N}(x) = (-1)^{\tfrac{N^2-N}{2}} \frac{ \det(e^{\lambda_i x_j})_{i,j=1}^{N}}{h(\lambda)}$$
where $h(\lambda) = \prod_{1\leq i<j\leq N} (\lambda_j-\lambda_i)$ is the Vandermonde determinant.

The expansion of the Gamma function near zero shows that as $\e$ goes to zero,
$$\prod_{m=1}^n \prod_{i=1}^N\frac{\Gamma(\e(\hat\theta_m+\lambda_i))}{\Gamma(\e(\theta_i+\hat\theta_m))} \to \prod_{m=1}^n \prod_{i=1}^N\frac{\theta_i+\hat\theta_m}{\hat\theta_m+\lambda_i}.$$

Likewise, the Euler Gamma reflection formula and the fact that $\sin (\pi x) / (\pi x)\to 1$ as $x\to 0$ yields
$$\e^{-N^2+N}s_N(\e\lambda) \to \frac{1}{(2\pi \iota)^N N!} (-1)^{\tfrac{N^2-N}{2}}  h(\lambda)^2,$$

where $h(\lambda) = \prod_{1\leq i<j\leq N} (\lambda_j-\lambda_i)$ is the Vandermonde determinant.

Putting these asymptotics together we have that the limit of our measure is given by

$$\prod_{i=1}^N dx_i \frac{\det(e^{-\theta_i x_j})_{i,j=1}^N}{h(\theta)} \frac{1}{(2\pi i)^N N!} \int_{\iota \R^N} d\lambda \det(e^{\lambda_i x_j})_{i,j=1}^{N} h(\lambda)\prod_{m=1}^n \prod_{i=1}^N \frac{\theta_i+\hat\theta_m}{\hat\theta_m+\lambda_i}.$$

For simplicity let us assume that $N=n$. Then this can be rewritten as

$$\frac{1}{Z_{N,N}}  \det(e^{-\theta_i x_j})_{i,j=1}^N D(x,\hat{\theta}) \prod_{i=1}^N dx_i$$
where we have
$$D(x,\hat{\theta})= \frac{1}{(2\pi i )^NN!}\int d\lambda \det(e^{\lambda_i x_j})_{i,j=1}^{N}  h(\hat{\theta}) h(\lambda) \prod_{m=1}^N \prod_{i=1}^N\frac{1}{\hat\theta_m+\lambda_i},$$
where the integrals are along lines parallel to the imaginary axis and to the right of the poles,
and where
$$Z_{N,N} = \frac{h(\theta)h(\hat{\theta})}{\prod_{i,m=1}^{N} (\theta_i+\hat{\theta}_m)} = \det(\frac{1}{\theta_i + \hat{\theta}_j})_{i,j=1}^{N}.$$

An application of the Residue Theorem provides that
$$D(x,\hat{\theta}) = \det(e^{-\hat{\theta}_i x_j})_{i,j=1}^{N} \prod_{i=1}^{N} {\bf 1}_{x_i\geq 0},$$
and hence our measure is simply
$$\frac{1}{Z_{N,N}}  \det(e^{-\theta_i x_j})_{i,j=1}^N \det(e^{-\hat{\theta}_i x_j})_{i,j=1}^{N} \prod_{i=1}^N dx_i {\bf 1}_{x_i\geq 0}$$
which coincides exactly with the formula given directly for last passage percolation and for the generalized Wishart ensemble in \cite{BP07}.

\subsection{Semi-discrete directed polymer in a Brownian environment}\label{semi-disc}

We will now indicate how an appropriate scaling of the solvable parameter matrix $(\gamma_{i,j})$ can be used to recover the
semi-discrete directed polymer in a Brownian environment studied in \cite{OConYor,MO,SV,OCon}, as a scaling limit of $z_{N1}(n)$.
In particular, we will consider the weight matrix $(d_{ij})$ with solvable parameter matrix $\gamma_{ij}=n$ and we will let $n$ tend to infinity.
We would first need some facts about the digamma and the trigamma functions, which are defined as
\be
\Psi_0(\theta)=\frac{d}{d\theta} \log \Gamma(\theta)\quad \text{and} \quad
\Psi_1(\theta)=\frac{d^2}{d\theta^2} \log \Gamma(\theta).
 \label{digam}\ee
In particular, we have that for $\theta$ large
\begin{equation}\label{gamma fact}
\Psi_0(\theta)=\log\theta-\frac{1}{2\theta}+o(\frac{1}{\theta}),\qquad \Psi_1(\theta)=\frac{1}{\theta}+o(\frac{1}{\theta}).
\end{equation}
We also notice that for an inverse Gamma random variable $d$ with parameter $\theta$ it holds that
\begin{eqnarray*}
\Psi_0(\theta)=-E[\log d] \quad \text{and} \quad \Psi_1(\theta) =\text{Var}(\log d).
\end{eqnarray*}
We write $\log  z_{N1}(n)$ in terms of the weights $d_{ij}$ (iid inverse Gamma random variables with parameter $n$) as
\begin{eqnarray*}
\log z_{N1}(n)&=&\log \sum_{1\leq i_1\cdots\leq i_N=n} \exp\left[ \sum_{j=1}^N\sum_{i_{j-1}\leq i\leq i_j}\log d_{ij}\right]\\
&=&
\log n^{-(N-1)} \sum_{1\leq i_1\cdots\leq i_N=n} \exp\left[ \sqrt{n\Psi_1(n)}\sum_{j=1}^N\frac{1}{\sqrt{n}}\sum_{i_{j-1}\leq i\leq i_j}\frac{\log d_{ij}+\Psi_0(n)}{\sqrt{\Psi_1(n)}}\right]\\
&&+(N-1)\log n -(n+N-1)\Psi_0(n).
\end{eqnarray*}
Using \eqref{gamma fact} we have that $n\Psi_1(n)\to 1$ and $(N-1)\log n -(n+N-1)\Psi_0(n) = -n\log n +\frac{1}{2} +o(1)$, as $n$ tends to infinity. It is now an easy consequence of Donsker's invariance principle and the Riemann  integral definition that
\begin{eqnarray*}
\log \big(n^n z_{N1}(n)\big)-\frac{1}{2}\Longrightarrow \log \int\cdots\int_{0\leq t_1\leq\cdots \leq t_{n-1}\leq 1}e^{\sum_{i=1}^N (B^i(t_i)-B^i(t_{i-1}))}dt_1\cdots dt_{n-1},
\end{eqnarray*}
for $n$ tending to infinity, with $B^i(\cdot)$ independent, standard Brownian motions. This is the directed polymer model studied in
\cite{OConYor,MO,SV,OCon}. Along the same lines it follows that the whole pattern $(z_{k\ell}(n))$ obeys similar scaling limits.

\section{Proof of main results}\label{proofsec}

\subsection{Proof of Proposition \ref{itpr1}}\label{itpr1proofsec}
We prove the intertwining relation   (\ref{itQKKPI})  in two steps. We use induction on $N$  to show that \eqref{itPKKPI} is valid. The inductive step will also reveal an inductive property of the eigenfunctions $\wfunc{N}{\thetacv}$.

\vskip 2mm
The  case $N=1$ in \eqref{itPKKPI}  is immediate since
$\Piker{1}{\thetarc^{[n]}}= \Pker{1}{\thetarc^{[n]}}$
and   $\Kker{1}{\thetac_1}$ is the identity operator.

A supporting  step of the proof is an intertwining that involves only two rows of the
array $\zarr$.
For this purpose introduce  two further kernels for $2\le k\le N\le n$: a time $n$ kernel $\Rker{k}{\thetarc^{[n]}}$ on $\SN{k-1}\times \SN{k}$ by
\be
\Rker{k}{\thetarc^{[n]}}((z^{k-1},z^k), d\tilde z^{k-1}, d\tilde z^k)
=  \Pker{k-1}{\thetarc^{[n]}}(z^{k-1}, d\tilde z^{k-1})\,
\Lker{k}{\thetarc_{n,k}}((z^{k-1},z^k; \tilde z^{k-1}), d\tilde z^k)
\ee
and
   a time-homogeneous kernel from $\SN{k}$ to $\SN{k-1}$ by
\begin{equation}
\Laker{k}{\thetacv}(y,dx)= \biggl\{\,\prod_{\ell=1}^{k-1}
\biggl(\frac{x_\ell}{y_\ell}\biggr)^{\thetac_k-\thetac_\ell} \biggr\}
 \exp\biggl[-\sum_{\ell=1}^{k-1} \left(\frac{x_\ell}{y_\ell} + \frac{y_{\ell+1}}{x_\ell}\right)\biggr]\,\prod_{\ell=1}^{k-1} \frac{dx_\ell}{x_\ell}.
\label{Laker}\end{equation}

\begin{lemma}\label{lm:2l-int}  At every time $n\ge 1$ and
for all $y\in \SN{N}$,  we have the following equality of
  measures on  $\SN{N-1}\times \SN{N}$, in terms  integration variables $(dz^{N-1}, dz^N)$:
\be \Pker{N}{\thetarc^{[n]}}(y, dz^N)\,\Laker{N}{\thetacv}(z^N, dz^{N-1})
 = \int_{\hat x\in\R_+^{N-1}} \Laker{N}{\thetacv}(y, d\hat x) \,
\Rker{N}{\thetarc^{[n]}}((\hat x,y), dz^{N-1}, d z^N). \label{2l-int} \ee
\end{lemma}

With \eqref{Laker} we can give  this alternative representation to the
  intertwining kernel  \eqref{Kker}
 from $\SN{N}$ to $\TN$:
\be \begin{aligned}
   \Kker{N}{\thetacv}(y,dz) &=    \delta_{y}(dz^N)
  \prod_{j=0}^{N-2} \Laker{N-j}{\thetacv}(z^{N-j}, dz^{N-j-1}).
 \label{Kker2}\end{aligned}
\ee
Postpone the proof of  Lemma \ref{lm:2l-int}  for a moment.
 With \eqref{Kker2} and  Lemma \ref{lm:2l-int}  we can   complete the proof of Proposition
\ref{itpr1} by checking the  induction step.  Assume \eqref{itPKKPI}  for $N-1$.
\begin{align*}
&\int_{\tilde y\in\R_+^N} \Pker{N}{\thetarc^{[n]}}(y, d\tilde y)\, \Kker{N}{\thetacv}(\tilde y, dz^{1,N})
=
\Pker{N}{\thetarc^{[n]}}(y, dz^N)
\prod_{j=0}^{N-2} \Laker{N-j}{\thetacv}(z^{N-j}, dz^{N-j-1})
\intertext{by Lemma \ref{lm:2l-int}}
&\qquad =\;\int_{\hat x\in\R_+^{N-1}} \Laker{N}{\thetacv}(y, d\hat x)
\,\Rker{N}{\thetarc^{[n]}}((\hat x, y), dz^{N-1}, dz^N)
\prod_{j=1}^{N-2} \Laker{N-j}{\thetacv}(z^{N-j}, dz^{N-j-1})
\intertext{by definition of $\Rker{N}{\thetarc^{[n]}}$}
&\qquad =\;\int_{\hat x\in\R_+^{N-1}} \Laker{N}{\thetacv}(y, d\hat x)
\,   \Pker{N-1}{\thetarc^{[n]}}(\hat x, dz^{N-1})
\prod_{j=1}^{N-1} \Laker{N-j}{\thetacv}(z^{N-j}, dz^{N-j-1})\\
&\qquad\qquad\qquad\qquad
 \times \Lker{N}{\thetarc_{n,N}}((\hat x, y; z^{N-1}), dz^N)
\intertext{by definition of $\Kker{N-1}{\thetac_{1:N-1}}$ and the induction assumption}
&\qquad =\;\int_{\hat x\in\R_+^{N-1}} \Laker{N}{\thetacv}(y, d\hat x)
\,  \int_{\tilde z^{1,N-1}\in T_{N-1}} \Kker{N-1}{\thetacv}(\hat x, d\tilde z^{1,N-1})
\, \Piker{N-1}{\thetarc^{[n]}}(\tilde z^{1,N-1}, dz^{1,N-1}) \\
&\qquad\qquad\qquad\qquad
 \times \Lker{N}{\thetarc_{n,N}}((\hat x, y; z^{N-1}), dz^N)
\intertext{by noting that $\tilde z^{N-1}=\hat x$  under
$\Kker{N-1}{\thetac_{1:N-1}}(\hat x, d\tilde z^{1,N-1})$, and by
definition of $\Piker{N}{\thetarc^{[n]}}$}
&\qquad =\;\int_{\hat x\in\R_+^{N-1}} \Laker{N}{\thetacv}(y, d\hat x)
\, \int_{\tilde z^{1,N-1}\in T_{N-1}} \Kker{N-1}{\thetacv}(\hat x, d\tilde z^{1,N-1})
\, \Piker{N}{\thetarc^{[n]}}((\tilde z^{1,N-1},y) , dz^{1,N})\\
&\qquad =\;
 \int_{\tilde z^{1,N}\in T_{N}} \Kker{N}{\thetacv}(y, d\tilde z^{1,N})
\, \Piker{N}{\thetarc^{[n]}}(\tilde z^{1,N} , dz^{1,N}).
\end{align*}
This checks \eqref{itPKKPI}  for $N$.

\begin{proof}[Proof of Lemma \ref{lm:2l-int}]
 Take a test function $g$ on $\SN{N-1}\times \SN{N}$, and collect and rearrange all the factors
in the integral against the kernel on   the right-hand side of \eqref{2l-int}:
\begin{align*}
&\int_{\R_+^{N-1}} \Laker{N}{\thetacv}(y, d\hat x) \,
\int_{\R_+^{N-1}\times\R_+^N}\Rker{N}{\thetarc^{[n]}}((\hat x,y), d\tilde x, d\tilde y)
\,g(\tilde x,\tilde y)\\
&= \biggl\{ \,\prod_{j=1}^{N}   \Gamma(\thetarc_{n,j})^{-1}\biggr\}\int_{\R_+^{N-1}\times\R_+} d\tilde x\; d\tilde y_1 \int_{\R_+^{N-1}} d\hat x
  \\
&\qquad\times
\exp\biggl[ -\sum_{\ell=1}^{N-1} \biggl(\frac{\hat x_\ell}{y_\ell} + \frac{y_{\ell+1}}{\hat x_\ell}
+  \frac{\hat x_\ell}{\tilde x_\ell}\biggr)
-\sum_{\ell=1}^{N-2} \frac{\tilde x_{\ell+1}}{\hat x_\ell} -\frac{y_1+\tilde x_1}{\tilde y_1}\,  \biggr]\\
&\qquad\times ({y_1+\tilde x_1})^{\thetarc_{n,N}} \, \tilde y_1^{\,-\thetarc_{n,N}-1}
\; \cdot\; \prod_{i=1}^{N-1}
\biggl(\frac{\hat x_i}{y_i}\biggr)^{\thetac_N-\thetac_i}
\;\cdot\; \prod_{j=1}^{N-1} \bigl(  \hat x_j^{\,\thetarc_{n,j}-1}  \tilde x_j^{\,-\thetarc_{n,j}-1} \bigr) \\
 &\qquad\qquad \times g\biggl(\tilde x\,, \,  \tilde y_1\,,\,
\biggl\{ \frac{y_{ \ell-1}\tilde x_{\ell-1}}{\hat x_{\ell-1}}\cdot
   \frac{y_{ \ell}+\tilde x_{\ell}}{y_{ \ell-1}+\tilde x_{\ell-1}} \biggr\}_{2\le\ell\le N-1} \,,\,
 \frac{y_{N}y_{N-1}\tilde x_{N-1}}{\hat x_{N-1}(y_{N-1}+\tilde x_{N-1})}   \,  \biggr) .
\end{align*}
Change variables in the inner integral from $\hat x$ to $\tilde y_{2,N}$ so that
the $g$-integrand becomes simply $g(\tilde x, \tilde y)$.
Recall that $\thetarc_{n,j}=\thetac_j+\thetar_n$.   After matching
up all the powers of the variables the integral above acquires the form
\begin{align*}
& \biggl\{ \,\prod_{j=1}^{N}   \Gamma(\thetarc_{n,j})^{-1}\biggr\}
\int_{\R_+^{N-1}\times\R_+^N} d\tilde x\; d\tilde y  \,
\exp\biggl[ -\sum_{\ell=1}^{N-1} \biggl(\frac{\tilde y_{\ell+1}}{y_\ell}
+ \frac{\tilde y_{\ell+1}}{\tilde x_\ell}
+  \frac{\tilde x_\ell}{\tilde y_\ell}\biggr)
-\sum_{\ell=1}^{N} \frac{y_\ell}{\tilde y_\ell}\,  \biggr]
\\
&\qquad\qquad\times \;\prod_{j=1}^{N-1}  \tilde x_j^{\,\thetac_N-\thetac_{j}-1}  \cdot
\prod_{j=1}^{N} \bigl(  y_j^{\,\thetarc_{n,j}} \; \tilde y_j^{\,-\thetarc_{n,N}-1} \bigr)
\cdot  g(\tilde x, \tilde y).
 \end{align*}
That this agrees with the integral coming from the left-hand side of \eqref{2l-int}
is just a matter of substituting in the explicit formulas  of the kernels.
The proof of Lemma \ref{lm:2l-int} is complete.
\end{proof}

\subsection{Proof of Theorem \ref{markovprojthm}}\label{markovprojthmproofsec}

Parts (i) and (ii) follow immediately from consequences (ii) and (i) (respectively) of Proposition \ref{MarkovfunctionsProp} when applied to $\Pbarker{N}{\thetarc^{[n]}}$, $\Kbarker{N}{\thetacv}$, and $\Piker{N}{\thetarc^{[n]}}$. In order to apply Proposition \ref{MarkovfunctionsProp} we must check two conditions. Condition (1) follows from Corollary \ref{itpr2} and condition (2) follows from the definition (\ref{Kbarker}) of $\Kbarker{N}{\thetacv}$.

Part (iii) follows from \eqref{Thm rel 2} setting $f(z)=\prod_{i=1}^N x_i(z)^{-\lambda_i}$ via the use of relations \eqref{Whittaker 1}, \eqref{Whittaker 2}, \eqref{Whittaker 3}, \eqref{Whittaker 4}:
\begin{eqnarray*}
\frac{1}{\prod_{i=1}^N y_i^{-\theta_i} w_{\thetacv}^N(y)} \int_{\TN}\prod_{i=1}^N x_i(z)^{-\lambda_i} M_{\thetacv}^N(y,dz)
&=&\frac{1}{\Psi^N_{\thetacv}(y)} \int_{\TN} \prod_{i=1}^N x_i(z)^{-\theta_i-\lambda_i} M^N(y,dz)\\
&=&\frac{\Psi^N_{\thetacv+\lambda}(y)}{\Psi^N_{\thetacv}(y)}
\end{eqnarray*}
In fact, if the $\lambda_i$ have non-zero real part then $f$ is not bounded. However if we set $f_n(x)=f(x){\bf 1}_{n^{-1}\leq x\leq n}$ then the claimed formula follows from dominated convergence and the fact that  $\int_{\TN} \Kbarker{N}{\thetacv}(y,dz)\, |f(z)|$ is bounded.

Part (iv), equation (\ref{Thm rel 3}) follows by integrating equation \eqref{Thm rel 2} over $\mu_n^N(y^0,dy)$ or using the intertwining relation (\ref{itPKKPI}) directly, as discussed in Remark~\ref{spec-remark}.

To deduce \eqref{int_mu} we  use the Plancherel formula \eqref{Planch}.
Let $\mu_n^N(y^0,y)$ denote the density of
  $\mu_n^N(y^0,dy)$. It exists because    $\mu_n^N(y^0,dy)$ comes from composing
  kernels with densities.
Multiply both sides of \eqref{Thm rel 3} by
$\int_{\RNplus} f(y)\Psi^N_\theta(y) \Psi^N_{-\lambda}(y) \prod_i {y_i}^{-1}{dy_i}$, integrate over $\iota \mathbb{R}^N$ with respect to $s_N(\lambda)d\lambda$ and use \eqref{Planch}. The application of the Plancherel identity is valid since $f(y)\Psi^N_\theta(y)$ is in $L^2(\RNplus,\prod_i dy_i/y_i)$ for $f$ continuous and  compactly supported on $\RNplus$,  and by the next lemma.

\begin{lemma}   Let $\thetarc_{i,j}=\thetar_i+\thetac_j>0$ be a solvable
parameter matrix.  Assume $\thetac_j-\thetar_i< 0<\thetar_i$
for all $i,j\ge 1$.
Then for  all  $y^0\in \RNplus$ and $n\ge 1$, the function
\begin{equation}
\frac{ \mu_n^N(y^0,y)}{\wf{N}{\theta}(y)}  \prod_{i=1}^{N} y_i
\end{equation}
is in $L^2(\RNplus,\prod_i dy_i/y_i)$.
\end{lemma}
\begin{proof}
Iterating definition \eqref{Pbarker} and using \eqref{Whittaker 2} gives
\[    \frac{\mu_n^N(y^0,y)}{\wf{N}{\theta}(y)} =
 \Big(\,\prod_{j=1}^{N} y_j^{\thetac_j}  \Big)\, \frac{p^N_n(y^0,y)}{\wfunc{N}{\theta}(y^0)} \]
 where  $p^N_n(y^0,y)$ is the density of the $n$-fold composition of  kernels
 $\Pker{N}{\thetarc^{[1]}}(y^0,dy^1),\dotsc,  \Pker{N}{\thetarc^{[n]}}(y^{n-1},dy)$
  from \eqref{PNdef}.    Let
\[   r_\alpha(u,v) =  \Gamma(\alpha)^{-1} \Bigl(\frac{u}{v}\Bigr)^{\alpha}
\,\frac{e^{-u/v}} {v} ,  \quad  u,v\in(0,\infty), \]
denote the transition density of a multiplicative random walk on $(0,\infty)$ with
$\Gamma^{-1}(\alpha)$-distributed steps, and let
\[   R_{n,j}(u^0,u^n)=  \int_{(0,\infty)^{n-1}}   \; \prod_{i=1}^{n} r_{\thetarc_{i,j}}(u^{i-1},\,u^{i})
\,du^{n-1}\dotsm du^1 \]
denote the  $n$-step transition  density with parameters $\thetarc_{i,j}$
 from column $j$ of our solvable matrix.
    By dropping the
killing term from \eqref{PNdef} and by an application of Jensen's inequality,
\[  \bigl(\,p^N_n(y^0,y)\,\bigr)^2  \le   \prod_{j=1}^{N}   \bigl(R_{n,j}(y_j^0,y_j)\bigr)^2
 \le   \prod_{j=1}^{N}
  \int_0^\infty  R_{n-1,j}(y_j^0,\tilde y_j)  \bigl(r_{\thetarc_{n,j}}(\tilde y_j,\,y_j)\bigr)^2 \,d\tilde y_j .    \]
  Put  these together, noting that the integrals factor over $\RNplus$:
 \begin{align*}
&\int_{\RNplus} \Big(\, {\prod_{j=1}^{N} y_j}\Big)
\biggl( \frac{\mu_n^N(y^0,y)}{\wf{N}{\theta}(y)}\biggr)^2 \,dy\\
&\qquad \le  \frac1{\wfunc{N}{\theta}(y^0)^2}
\; \prod_{j=1}^{N}  \;
  \int_{(0,\infty)^2}     y_j^{2\thetac_j+1}   R_{n-1,\,j}(y_j^0,\tilde y_j)  \bigl(r_{\thetarc_{n,j}}(\tilde y_j,\,y_j)\bigr)^2 \,dy_j\,d\tilde y_j \\
&\qquad =  \frac{C_{N,n}(\thetar^{[n]}, \thetac)}{\wfunc{N}{\theta}(y^0)^2} \;
 \prod_{j=1}^{N}  \;
  \int_{(0,\infty)}    \,   \tilde y_j^{2\thetac_j}  R_{n-1,\,j}(y_j^0,\tilde y_j)   \,d\tilde y_j \\
  &\qquad =  \frac{C_{N,n}(\thetar^{[n]}, \thetac)}{\wf{N}{\theta}(y^0)^2} \;
 \prod_{j=1}^{N}  \prod_{i=1}^{n-1} \mathbb E[ d_{i,j}^{2\thetac_j} ]  \  < \  \infty.
  \end{align*}
  The first equality above integrates away the variables $y_j$, and the finiteness
 of the constant  $C_{N,n}(\thetar^{[n]}, \thetac)$ depends on $\thetar_n>0$.
 The second equality uses the independence of random walk   steps.
 The finiteness of the expectations is equivalent to $\thetac_j-\thetar_i<0$.
\end{proof}

\subsection{Proof of Theorem \ref{main corollary}}\label{main corollaryproofsec}
To prove part (i) set $y^0=y^{0,M}$ in \eqref{int_mu}. Asymptotic relation (20) in~\cite{OCon} gives that, as $M\to\infty$, the ratio $ \Psi_{\lambda}^N(y^{0,M}) / \Psi_{\thetacv}^N(y^{0,M}) \to 1$. Therefore, we only need to demonstrate that the limit $M\to\infty$ can be passed inside the integral. This follows from dominated convergence since
\[\left| \frac{\Psi_\lambda^N(y^{0,M})}{\Psi_{\thetacv}^N(y^{0,M})}\right|\leq
\left| \frac{\Psi_0^N(y^{0,M})}{\Psi_{\thetacv}^N(y^{0,M})}\right|,
\]
for $\lambda\in \iota\mathbb{R}^N$ and the fact that the rest of the integrand is in $L^1(\iota\mathbb{R}^N,s_N(\lambda)d\lambda)$. The latter follows from the fact that $\int_{\RNplus}f(y)\Psi_{\thetacv}(y)\Psi_{-\lambda}(y) \prod_idy_i/y_i $$\in L^2(\iota\mathbb{R}^N,s_N(\lambda)d\lambda)$, by the Plancherel isomorphism and the fact that $f$ is bounded and compactly supported, and $\prod_m\prod_i \Gamma(\hat\theta_m+\lambda_i)/\Gamma(\theta_i+\hat\theta_m)\in L^2(\iota\mathbb{R}^N,s_N(\lambda)d\lambda) $, for $n\geq N$. Indeed, using the asymptotics
\[\lim_{x_2\to\infty}|\Gamma(x_1+\iota x_2)| e^{\frac{\pi}{2}|x_2|}|x_2|^{\frac{1}{2}-x_1}=\sqrt{2\pi},
\qquad x_1,x_2\in \mathbb{R}
\]
it follows that
\begin{eqnarray*}
 \left |\prod_{m=1}^n\prod_{i=1}^{N}\frac{ \Gamma(\hat\theta_m+\lambda_i)}{\Gamma(\theta_i+\hat\theta_m)}\right|^2 s_N(\lambda)
&\sim&  e^{-\pi n \sum_{i=1}^N|\lambda_i| +\frac{\pi}{2}\sum_{1\leq i\neq j \leq N}|\lambda_i-\lambda_j|}\\
&\lesssim& e^{-\pi n \sum_{i=1}^N|\lambda_i| +\pi(N-1)\sum_{i=1}^N|\lambda_i|},
\end{eqnarray*}
which decays exponentially when $n\geq N$.

Part (ii) follows from the Whittaker integral identity (\ref{bs})
 \[
 \int_{\RNplus} e^{-sy_1}\Psi_\theta(y)\Psi_{-\lambda}(y) \prod_i\frac{dy_i}{y_i} = s^{\sum(-\lambda_i+\theta_j)} \prod_{i,j}\Gamma(\lambda_i-\theta_j)\in L^2(\iota\mathbb{R}^N,s_N(\lambda)d\lambda).
 \]
We can now repeat the argument used in the proof of part (i), using the function $f(y)=e^{-sy_1}$ (instead of a compactly supported $f$).

To prove part (iii) we
show  in Proposition \ref{M-lim-pr} below a more general statement:
  the entire array
$\zarr(n)$ converges in distribution,  as $M\to\infty$,   to the one defined by the path configurations, noting that this
statement can make sense  only for the
portion of the array  constructed by time $n$.

Let $P^\zarr$ denote the probability distribution of the process $\zarr(\cdot)$ when the
initial state is $z(0)=z\in\TN$, and let $E^z$ denote expectation under $P^\zarr$.
Let us also use the notation $E^\es$ when the array starts empty, in which
case at time $n$ only the portion $\{\zarr_{k\ell}(n): 1\le k\le N,\,1\le\ell\le k\wedge n\}$ of the array has been
defined.  Recall   $y^{0,M}=\left(e^{-M\rho_{N,\ell}}\right)_{1\leq \ell \leq N}$ with
$\rho$ from \eqref{rho}.

\begin{proposition}  Let $N, n\ge 1$, and  let $f$ be a bounded continuous function
of the $(0,\infty)$-valued coordinates  $\{\zarr_{k\ell}(s): 1\le s\le n, \,1\le k\le N,\,1\le\ell\le k\wedge s\}$.  Then
\be
\lim_{M\to\infty}  \int_{\TN} \Kbarker{N}{\theta}(y^{0,M}, dz) \, E^z[ f( \zarr(1),\dotsc, \zarr(n)) ]
=   E^\es[ f(\zarr(1),\dotsc,  \zarr(n)) ].
\label{M-lim1}\ee
\label{M-lim-pr}\end{proposition}

Before turning to the proof of Proposition \ref{M-lim-pr} we use it to derive
part (iv) of Theorem \ref{main corollary}.   Let $\nu^M(dz)$ denote the initial
distribution $\Kbarker{N}{\theta}(y^{0,M}, dz)$ on arrays, and let  $f$ be a bounded continuous function
on arrays.   Then part (iv) follows from this calculation:
\begin{align*}
E^\es[ f( \zarr(n)) ] &= \lim_{M\to\infty}   E^{\nu^M}[ f( \zarr(n)) ]
= \lim_{M\to\infty}  E^{\nu^M} \int_{\TN} \Kbarker{N}{\theta}(y(n), dz) \, f(z)\\
&= E^\es \int_{\TN} \Kbarker{N}{\theta}(y(n), dz) \, f(z)
= \int_{\SN{N}} \mu^N_n(dy) \int_{\TN} \Kbarker{N}{\theta}(y, dz) \, f(z).
 \end{align*}
 The first and third equalities are instances of  \eqref{M-lim1}, the second is \eqref{Thm rel 1},  and the last one
 is Proposition \ref{M-lim-pr} again because one consequence of  limit
 \eqref{M-lim1},  together with  $\mu^N_n(y^{0,M}, dy)\to \mu^N_n(dy)$ from part (i),
 is that  $\mu^N_n$ is the marginal distribution of $y(n)$ under $P^\es$ for $n\ge N$.
 The third equality above is justified by arguing that
$\int\Kbarker{N}{\theta}(y, dz) \, f(z) $ is a continuous function of $y$, or
equivalently, that  $y\mapsto \Kbarker{N}{\theta}(y, dz)$ is a continuous mapping
into the space of probability measures on arrays (in the usual weak topology of
probability measures, generated by bounded continuous functions).
This follows from the fact that, off the bottom row, $\Kbarker{N}{\theta}(y, dz)$ has
a density that is  jointly continuous  in $(y,z)$.  Pointwise convergence of
densities implies convergence of probability measures, a result known as
Scheff\'e's theorem.   This completes the proof of part (iv).

To prove part (v) of Theorem \ref{main corollary} we will use the alternative form of
the Whittaker integral identity (\ref{bs'}).  By part (i) of the theorem, we have
\begin{eqnarray*}
\int_{\mathbb{R}^N_+} f(y) \mu_N^N(dy)=
\int_{\iota \mathbb{R}^N} d\lambda s_N(\lambda)
 \left( \int_{\mathbb{R}^N_+} f(y) \Psi^N_\theta(y) \Psi^N_{-\lambda}(y) \prod_i\frac{dy_i}{y_i}\right)
 \prod_{m=1}^N\prod_{i=1}^N\frac{\Gamma(\hat\theta_m+\lambda_i)}{\Gamma(\theta_i+\hat\theta_m)}.
\end{eqnarray*}
for any continuous, compactly supported function $f$ on $\RNplus$.  By (\ref{bs'})
$$\int _{\RNplus}e^{-y_N^{-1}} \wf{N}{\lambda}(y)\wf{N}{\hat\theta}(y)\prod_{i=1}^N\frac{dy_i}{y_i} =
 \prod_{m=1}^N\prod_{i=1}^N\Gamma(\lambda_i+\hat\theta_m).$$
 As above, the functions $\int_{\mathbb{R}^N_+} f(y) \Psi^N_\theta(y) \Psi^N_{-\lambda}(y) \prod_i\frac{dy_i}{y_i}$
 and $ \prod_{m=1}^N\prod_{i=1}^N\Gamma(\lambda_i+\hat\theta_m)$ are both in $L^2(\iota\R^N,s_N(\lambda)d\lambda)$
 so we have, by the Plancherel theorem,
 $$\int_{\mathbb{R}^N_+} f(y) \mu_N^N(dy)=\prod_{m=1}^N\prod_{i=1}^N\Gamma(\theta_i+\hat\theta_m)^{-1}
  \int_{\mathbb{R}^N_+} f(y) \Psi^N_\theta(y) e^{-y_N^{-1}} \Psi^N_{\hat\theta}(y) \prod_{i=1}^N\frac{dy_i}{y_i},$$
  as required.

\begin{proof}[Proof of Proposition \ref{M-lim-pr}]
 Figure \ref{NYfig3} shows that  $\{\zarr_{k\ell}(s): 1\le s\le n, \, 1\le k\le N,\,1\le\ell\le k\wedge s\}$
 can be written as a function of $\{\zarr_{m+1}(m): 0\le m< n\wedge N\}$ and $d^{[1,n]}$.
 Let $( \zarr(1),\dotsc, \zarr(n))=G\bigl( (\zarr_{m+1}(m))_{0\le m<n\wedge N}, d^{[1,n]}\bigr)$ represent this
 functional relationship defined by the row insertion procedure.   The case of starting with an empty array is  the one where each vector $\zarr_{m+1}(m)=e_1^{(N-m)}$ where
 $e_1^{(k)}=(1,0,\dotsc, 0)$ is the first  $k$-dimensional   unit vector.  This can be
 seen from Figure \ref{NYfigb}.
   If we let $\mathbb P$ denote the probability distribution of the
 weight matrix $d$,  the goal \eqref{M-lim1}  can be re-expressed as
\be \begin{aligned}
& \lim_{M\to\infty}  \int_{\TN} \Kbarker{N}{\theta}(y^{0,M}, dz)  \int_{\R_+^{nN}} \mathbb P(d(d^{[1,n]})) \,
f\bigl(G\bigl( (\zarr_{m+1}(m))_{0\le m< n\wedge N}, d^{[1,n]}\bigr)\bigr)  \\
&\qquad\qquad \qquad
 =    \int_{\R_+^{nN}}  \mathbb P(d(d^{[1,n]})) \, f\bigl(G\bigl( ( e_1^{(N-m)} )_{0\le m< n\wedge N}, d^{[1,n]}\bigr)\bigr).
\end{aligned} \label{M-lim2}\ee
(Notation $\mathbb P(d(d^{[1,n]}))$   means that the matrix $d^{[1,n]}$ is the integration
variable under the measure $\mathbb P$.)
On the left above the vectors $(\zarr_{m+1}(m))_{0\le m< n\wedge N}$ are themselves functions of the initial
values
 $(\zarr_m(0))_{0\le m< n\wedge N}=(\zarr_m)_{0\le m< n\wedge N}$ and the weights  $d^{[1,n-1]}$, as shown in Figure \ref{NYfig3}.   Comparison of the $\xi'$ output
 in equations \eqref{g-row-ins} and  \eqref{es-row-ins}
 shows that the mapping $G$ is continuous as the inputs $\zarr_{m+1}(m)\to e_1^{(N-m)}$.
Thus the upshot is that we need to show the weak convergence
$(\zarr_{m+1}(m))_{0\le m< n\wedge N} \to  (e_1^{(N-m)} )_{0\le m< n\wedge N}$ as $M\to\infty$.
Since the limit is deterministic we can ignore the joint distribution and do this one
coordinate at a time. So
it suffices to fix $0\le m<k\le N$ such that $m<n\wedge N$  and show that
\be
\zarr_{k,m+1}(m) \to  \delta_{k,m+1}  \quad\text{ in probability  as $M\to\infty$}  \label{M-lim3}\ee
when $\zarr_{k,m+1}(m)$ has the probability distribution described by the left-hand side of \eqref{M-lim2} and $\delta_{k,m+1}$ is the Kronecker delta.

Write
\be \zarr_{k,m+1}(m)=V_{k,m}[(\zarr_{k,\ell}(0))_{1\le\ell\le m+1, \ell\le k\le N}, d^{[1,m]}] \label{zV}\ee
to indicate the functional relationship from the inputs to the array element
$\zarr_{k,m+1}(m)$.   Our   goal \eqref{M-lim3} follows if we show that
for any fixed $d^{[1,m]}\in(0,\infty)^{mN}$ and a bounded continuous test function $f$,
\be \begin{aligned}
& \lim_{M\to\infty}
 \int_{\TN} \Kbarker{N}{\theta}(y^{0,M}, dz)
f\bigl( V_{k,m}[ (\zarr_{k,\ell})_{1\le\ell\le m+1, \ell\le k\le N}, d^{[1,m]}]\bigr)
 =   f( \delta_{k,m+1}).
\end{aligned} \label{M-lim12}\ee

To understand the asymptotics of  $\Kbarker{N}{\theta}(y^{0,M}, dz) $ it is convenient
to switch from multiplicative to additive variables. Define an array
$t=\{t_{k\ell}\}_{1\le \ell\le k\le N}$ by  $\zarr_{k\ell}=e^{t_{k\ell}}$.
 Let $\wt K(y,\,dt)$ denote the distribution of the array $t$ when $\zarr$ has distribution
 $\Kbarker{N}{\theta}(y, d\zarr)$.  For
 $u\in\R^N$ let
 \[ W(u)=\{ t=\{t_{k\ell}\}_{1\le \ell\le k\le N}: t_{N,i}=u_i, \, 1\le i\le N\} \]
 be the set of arrays
 with  bottom row  $u$.
Let
\be
\mathcal{F}_{\theta} (t)= \sum_{k=1}^N\thetac_k \Bigl(\; \sum_{\ell=1}^{k-1} t_{k-1,\ell}
- \sum_{\ell=1}^kt_{k,\ell} \Bigr)
                  -\sum_{k=1}^{N-1}\sum_{\ell=1}^k \left( e^{t_{k,\ell}-t_{k+1,\ell}} +
                                                        e^{t_{k+1,\ell+1}-t_{k,\ell}}\right).
\label{cF}\ee
Then, for a bounded continuous test function $g$,
\be\begin{aligned}
  \int g(t) \,\wt K(y^{0,M},\,dt) &=  \frac1{C(M)}
\int_{W(-M\rho^{[n]}) }  g(t) \, e^{\mathcal F_\theta(t)}\,dt  \\
&= \frac1{C(M)} \int_{W(0)}   g(t-M\rho) \, e^{S_\theta(t)+e^{M/2}\mathcal F_0(t) }\,dt .
\end{aligned}\label{M7}\ee
Above $C(M)=\int_{W(0)}    e^{S_\theta(t)+e^{M/2}\mathcal F_0(t) }\,dt $ is the  normalization needed for a probability measure.
We changed  variables by shifting  $t$ to $t-M\rho$ where $\rho=(\rho_{k\ell})_{1\le \ell\le k\le N}$
is the array from \eqref{rho} defined  by $\rho_{k\ell}=\tfrac12(k-1)-\ell+1$.
 Defining  $S_\theta(t)=\mathcal F_\theta(t)-\mathcal F_0(t)$   leads to
 \[ \mathcal F_\theta(t-M\rho)= S_\theta(t-M\rho)+\mathcal F_0(t-M\rho)
=S_\theta(t)+e^{M/2}\mathcal F_0(t).\]

  Return to the right-hand side of \eqref{M-lim12}
to rewrite as
 \begin{align}
& \int_{\TN} \Kbarker{N}{\theta}(y^{0,M}, dz)
f\bigl( V_{k,m}[ (\zarr_{k,\ell})_{1\le\ell\le m+1,\, \ell\le k\le N}, d^{[1,m]}]\bigr)  \nn\\
&=  \int  \wt K(y^{0,M}, dt)
f\bigl( V_{k,m}[ (e^{t_{k\ell}} )_{1\le\ell\le m+1,\, \ell\le k\le N}, d^{[1,m]}]\bigr) \nn\\
&= \frac1{C(M)} \int_{W(0)}  f\bigl( V_{k,m}[ (e^{t_{k\ell}-M\rho_{k\ell}} )_{1\le\ell\le m+1,\, \ell\le k\le N}, d^{[1,m]}]\bigr)  \, e^{S_\theta(t)+e^{M/2}\mathcal F_0(t) }\,dt.
\label{aux-27}\end{align}
We claim that
\be  \text{line \eqref{aux-27}   converges to $ f( \delta_{k,m+1})$ as $M\to\infty$.}
\label{M-lim12.5}\ee
Limit \eqref{M-lim12.5} finishes the proof of Proposition \ref{M-lim-pr}.  To establish it we prove
 the two lemmas below.

\begin{lemma}  On the set $W(0)$, the function  $\mathcal F_0$ is strictly concave and  has a unique
maximum $t^0$ that satisfies   $\sum_{\ell=1}^k t^0_{k\ell}=0$ for each $1\le k\le N$.
\label{F-lm}\end{lemma}

\begin{lemma}  For each fixed  $1\le m+1\le k\le N$, $m<n$, $d^{[1,m]}\in(0,\infty)^{mN}$,
\be
\lim_{\substack{M\to\infty\\ t\to t^0}} \zarr_{k,m+1}(m)=
\lim_{\substack{M\to\infty\\ t\to t^0}}  V_{k,m}[ (e^{t_{k\ell}-M\rho_{k\ell}})_{1\le\ell\le m+1, \ell\le k\le N}, d^{[1,m]}]
= \delta_{k,m+1}.
\label{M-lim13}\ee
\label{M-lim13-lm}\end{lemma}

Lemma \ref{F-lm} implies that the probability
measure $C(M)^{-1} \mathbf 1_{W(0)}(t)\,e^{S_\theta(t)+e^{M/2}\mathcal F_0(t) }\,dt$ converges
weakly  to the pointmass at $t^0$.   This together with
  \eqref{M-lim13} and the boundedness and continuity of $f$ imply   \eqref{M-lim12.5}.
We have proved Proposition \ref{M-lim-pr} but it remains to prove the lemmas above.
\end{proof}

\begin{proof}[Proof of Lemma \ref{F-lm}]   This lemma comes from \cite{OCon}
and \cite{R}. We include the proof for the sake of completeness.

The critical point equations    $\frac\partial{\partial t_{ki}}\mathcal F_0(t) =0$
for $1\le i\le k<N$  rearrange to
\begin{align*}
e^{2t_{ki}}=\frac{  e^{t_{k-1,i}} \ind_{\{i<k\}} + e^{t_{k+1,i+1}}  }
{ e^{-t_{k-1,i-1}} \ind_{\{k\ge i>1\}} + e^{-t_{k+1,i}} }     \qquad  \text{for $1\le i\le k<N$ }.
 \end{align*}
The case $k=1$ gives   $t_{11}=(t_{21}+t_{22})/2$.  Multiplying  equations
together  gives
\[    e^{2\sum_{i=1}^k t_{ki}}=  e^{\sum_{i=1}^{k-1} t_{k-1,i}  + \sum_{i=1}^{k+1}{t_{k+1,i+1}} }
\quad \text{ for $2\le k\le N$.}  \]
From this follows   $k^{-1}\sum_{i=1}^k t_{ki}=N^{-1}\sum_{i=1}^k t_{Ni}$ for
$1\le k<N$, and these all $=$ $0$ by the $W(0)$ condition.

Following \cite[p.~136]{R} we  write $\mathcal F_0$ in the following form.
Consider  the directed graph $(\mathcal V, \mathcal E)$
with vertex set $\mathcal V=\{ (k,i):  1\le i\le k\le N\}$ and where
$\mathcal E$ contains  all possible  edges $((k+1,i),(k,i))$
and $((k,i),(k+1,i+1))$.   Edge $a=(u(a), v(a))$ is directed form vertex $u(a)$ to vertex $v(a)$.
Then
 \[  \mathcal F_0(t) = -\sum_{a\in\mathcal E} e^{t_{v(a)}-t_{u(a)}}  \]
and  for vertices $x,y$
\begin{align*}   \frac{\partial^2 \mathcal F_0(t)  }{\partial t_{x} \partial  t_{y}}
&=-\,\biggl( \; \sum_{a\in\mathcal E: v(a)=x} e^{t_x-t_{u(a)}}  +
   \sum_{a\in\mathcal E: u(a)=x} e^{t_{v(a)}-t_{x}}    \biggr)\ind_{\{x=y\}}
 \;  +  e^{t_x-t_y} \ind_{\{(y,x)\in\mathcal E\}} + e^{t_y-t_x} \ind_{\{(x,y)\in\mathcal E\}}.
\end{align*}
Take a vector  $(\alpha_x)_{x\in\mathcal V}$ such that $\alpha_{(N,i)}=0$ (because
the variables $t_{N,i}$ are not free to vary on $W(0)$).
Then
\begin{align*}  \sum_{x,y\in\mathcal V} \alpha_x\alpha_y \frac{\partial^2 \mathcal F_0(t)  }{\partial t_{x} \partial  t_{y}}
= -  \sum_{a\in\mathcal E} (\alpha_{v(a)}-\alpha_{u(a)})^2 e^{t_{v(a)}-t_{u(a)}}
\end{align*}
is $<0$ unless $\alpha=0$.  This gives strict concavity and the unique maximum.
 \end{proof}

\begin{proof}[Proof of Lemma \ref{M-lim13-lm}]   First we take care of the case
$m=0$.   This is read off directly from the initial values and $t^0_{11}=0$:
$z_{k,1}(0)=e^{t_{k,1}+M(1-(k+1)/2)}\to \delta_{k,1}$ as $M\to\infty$ and
$t\to t^0$.

For the rest of the proof $m\ge 1$.  We turn to the matrix machinery
developed in \cite{NY}.  For that purpose we  consider the
  row insertion procedure also   in terms of ratio variables.
   Let $\zratio=(\zratio_\ell,\dotsc, \zratio_N)$ denote the ratio variables
associated with the vector  $\xi=(\xi_\ell,\dotsc, \xi_N)$:
 $\xi_k=\zratio_\ell \zratio_{\ell+1}\dotsm \zratio_k$  for $k=\ell,\dotsc, N$.
 Similarly   $\xi'_k=\zratio'_\ell \zratio'_{\ell+1}\dotsm \zratio'_k$.
 Then the row insertion
\be
\begin{array}{ccc}
& b & \\
\xi & \cross & \xi' \\
& b'
\end{array}
\quad \text{  defined in Definition \ref{NYdef}  is equivalently expressed as} \quad
\begin{array}{ccc}
& b & \\ \zratio & \cross & \zratio'  . \\
& b'
\end{array} \label{rowins4}\ee
Recall  definition \eqref{Hm} of the $N\times N$ matrices $H_m(\zratio)$.
Then  \eqref{rowins4} is equivalent to    \cite[eqn.~(2.23)--(2.25)]{NY}
\be H_\ell(\zratio)H_\ell(b)=H_{\ell+1}(b')H_\ell(\zratio') .  \label{rowins5}\ee
In the extreme case $\ell=N$ there is no $b'$ left and the correct
interpretation is $H_{N+1}(b')=I=$ the $N\times N$ identity matrix.

As in the end of Section \ref{Kirillov corres}, define the ratio variables of the arrays  by
\be \text{ $\zratio_{\ell\ell}(n)=z_{\ell\ell}(n)$ and
$\zratio_{k\ell}(n)=z_{k\ell}(n)/z_{k-1,\ell}(n)$ for $1\le \ell<k$. }\label{xz}\ee

Applying \eqref{rowins5} to
  the  upper left corner of Figure \ref{NYfig3} gives
  \[ H_1(\zratio_1(0))H_1(a_1(1))=H_2(a_2(1))H_1(\zratio_1(1)).\]
  Left multiply this identity
by $H_2(\zratio_2(0))$, $H_3(\zratio_3(0))$, $\dotsc$,  right multiply by $H_1(a_1(2))$, $H_1(a_1(3))$,
$\dotsc$, and apply \eqref{rowins5}
repeatedly on the right-hand side.  This gives the following identity for all $m\ge 1$:
\be
\prod_{i=0}^m H_{m+1-i}(\zratio_{m+1-i}(0)) \cdot \prod_{j=1}^m  H_1(a_1(j))
= \prod_{j=1}^m  H_{m+2}(a_{m+2}(j)) \cdot  \prod_{i=0}^m H_{m+1-i}(\zratio_{m+1-i}(m))
\label{Heq1}\ee

If $m=N-1$ then $H_{m+2}(a_{m+2}(j))=I$ and the first product on the right disappears.
If $m<N-1$ then apply \cite[Thm.~2.4]{NY}  to the lower right $(N-m-1)\times(N-m-1)$ block of the
first product on the right.  This gives vectors $p^{m+\ell+1},\dotsc, p^{m+2}$ such that  $p^i=(p^i_i,\dotsc, p^i_N)$,
\[ \prod_{j=1}^m  H_{m+2}(a_{m+2}(j)) =  \prod_{j=0}^{\ell-1}  H_{m+\ell+1-j}(p^{m+\ell+1-j}),  \]
and $\ell=m\wedge(N-m-1)$.
Substituting this back into \eqref{Heq1} gives
\be
\prod_{i=0}^m H_{m+1-i}(\zratio_{m+1-i}(0)) \cdot \prod_{j=1}^m  H_1(d^{[j]})
= \prod_{j=0}^{\ell-1}  H_{m+\ell+1-j}(p^{m+\ell+1-j}) \cdot  \prod_{i=0}^m H_{m+1-i}(\zratio_{m+1-i}(m)).
\label{Heq2}\ee
Let $H$ denote the matrix on the left. On the right we  have a descending sequence of subscripts $(2m+1)\wedge N,\dotsc,1$.  We can appeal to \cite[Prop.~1.6]{NY}  to conclude that the vectors on the right-hand side are uniquely determined.   In particular, $\zratio_{m+1}(m)=(\zratio_{k,m+1}(m))_{k=m+1}^N$   is given by
\be
\zratio_{m+1,m+1}(m)  = \frac{\tau^{m+1}_{m+1}}{\tau^{m}_{m+1}} \,,  \quad
 \zratio_{k,m+1}(m) = \frac{\tau^{m+1}_{k}\tau^{m}_{k-1}}{\tau^{m}_{k}\tau^{m+1}_{k-1}}
 \quad\text{ for $m+1<k\le N$,}
 \label{x9}\ee
where  $\tau^i_j=\det H^{[1,i]}_{[j-i+1,j]}$, $i\le j$,  are minor determinants of the matrix $H$
over rows  $1,\dotsc, i$ and columns $j-i+1, \dotsc, j$.
Switching back to $\zarr$ via \eqref{xz} gives
\be
  \zarr_{k,m+1}(m) = \frac{\tau^{m+1}_{k}}{\tau^{m}_{k}}
 \quad\text{ for $m+1\le k\le N$.}
 \label{x9.1}\ee
This is the function $V_{k,m}$ defined in \eqref{zV}.
(Prop.~1.6 of \cite{NY} needs a hypothesis on the minors of $H$.  This hypothesis can be
checked from \eqref{tau-2} below with the help of Figure \ref{NYfig7}.)

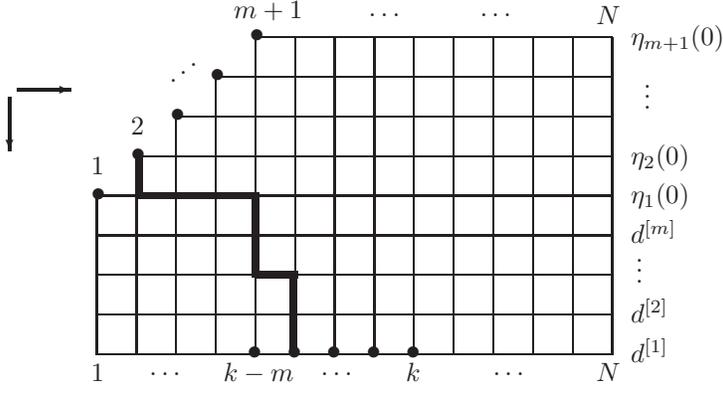
\begin{figure}
\begin{center}
\begin{picture}(200,120)(0,-15)
\multiput(0,0)(0,15){5}{\line(1,0){195}}
\put(202, -3){$d^{[1]}$}
\put(202, 12){$d^{[2]}$}   \put(204, 27){$\vdots$}
\put(202, 42){$d^{[m]}$}
\put(202, 57){$\zratio_1(0)$} \put(202, 72){$\zratio_2(0)$}  \put(207, 93){$\vdots$}
 \put(202, 117){$\zratio_{m+1}(0)$}
\put(15,75){\line(1,0){180}}
\put(30,90){\line(1,0){165}}  \put(45,105){\line(1,0){150}}\put(60,120){\line(1,0){135}}
\put(0,0){\line(0,1){60}}  \put(15,0){\line(0,1){75}}  \put(30,0){\line(0,1){90}} \put(45,0){\line(0,1){105}}
\multiput(60,0)(15,0){10}{\line(0,1){120}}
\multiput(-2,58)(15,15){5}{\large$\bullet$}
\put(-2,68){$1$} \put(13,83){$2$} \put(27, 102){$\iddots$}
 \put(52,127){$m+1$} \put(103, 128){$\dots$} \put(145, 128){$\dots$}  \put(189,125){$N$}
 \multiput(57,-2)(15,0){5}{\large$\bullet$}
\put(-2,-10){$1$}  \put(20, -8){$\dots$}
\put(48,-10){$k-m$}  \put(85, -8){$\dots$}  \put(117,-10){$k$}   \put(150, -8){$\dots$}
 \put(189,-10){$N$}
{\linethickness{1.0pt}\put(-30,100){\vector(1,0){20}}
\put(-33,97){\vector(0,-1){20}}  }
{\linethickness{2.5pt}
\put(16,75){\line(0,-1){15}} \put(15,60){\line(1,0){46}}  \put(60,61){\line(0,-1){32}}
\put(60,30){\line(1,0){16}} \put(74.5,30){\line(0,-1){30}}
}
\end{picture}
\caption{The minor $\tau^{m+1}_k$ equals the sum of the weights of $(m+1)$-tuples
of disjoint down-right paths  $(\gamma^1,\dotsc, \gamma^{m+1})$, where
$\gamma^j$ is a path from vertex $j$ at the top to vertex $k-m-1+j$ at the bottom.
The thickset line displays one  admissible path $\gamma^2$ from top vertex $2$
to bottom vertex $k-m+1$.} \label{NYfig7}
\end{center} 
\end{figure}

We use a graphical representation to compute the minors $\tau^i_j$, in the spirit of the Lindstr\"om-Gessel-Viennot method, following Sect.~1.1 of \cite{NY}.
 The matrix $H$ is represented by an array of $2m+1$ right-adjusted rows, one row for each vector $\zratio_{m+1}(0),\dotsc, \zratio_{1}(0), d^{[1]}, \dotsc, d^{[m]}$ (see Figure \ref{NYfig7}).
For $1\le i\le m+1$, the vertices on row $i$ are assigned weights $\zratio_{m+2-i, m+2-i}(0), \dotsc, \zratio_{N, m+2-i}(0)$,  and for $m+2\le i\le 2m+1$, the vertices on row $i$ are assigned weights  $d_{i-m-1,1},\dotsc, d_{i-m-1,N}$. Note that due to initial elements missing from the $\eta$-vectors,  the top vertex of column $j$ is on row $m-j+2$ for $1\le j\le m$.    Combining (1.16) and (1.27) in \cite{NY} gives
\be
\tau^i_j=\det H^{[1,i]}_{[j-i+1,j]} = \sum_{(\gamma^1,\dotsc,\gamma^i)}
wt(\gamma^1,\dotsc,\gamma^i)  \label{tau-2}\ee
where the sum ranges over $i$-tuples of disjoint paths $\gamma^1,\dotsc,\gamma^i$ such that $\gamma^k$ goes from vertex $k$ at the top edge of the graph to vertex $j-i+k$ at the bottom edge, and the weight   $wt(\gamma^1,\dotsc,\gamma^i)$ of the $i$-tuple is the product of the weights on the vertices of the paths.

Now we find the asymptotics of the minors in \eqref{x9.1}.  We are proving \eqref{M-lim13} so the initial   $z_{k\ell}$-values are $z_{k\ell}=e^{t_{k\ell}-M\rho_{k\ell}}=e^{t_{k\ell}+M(\ell-(k+1)/2)}$.
From this we get the initial ratio variables
\[  \text{$\zratio_{\ell\ell}=z_{\ell\ell}=e^{t_{\ell\ell}+M(\ell-1)/2}$, and
$\zratio_{k\ell}=z_{k\ell}/z_{k-1,\ell}=e^{t_{k\ell}-t_{k-1,\ell}-M/2}$   for $k>\ell$.} \]
Note in particular that on the top $m+1$ rows of the array in Figure \ref{NYfig7}, all but the left edge weights  decay as $Ce^{-M/2}$.

Consider first $\tau^{m+1}_k$ for some $k>m+1$. One can check by induction on $m$ that every $(m+1)$-tuple of paths $(\gamma^1,\dotsc, \gamma^{m+1})$
 from vertices  $(1,\dotsc,m+1)$ on the top edge to vertices  $(k-m,\dotsc,k)$ on the bottom edge contains at least $m(m+1)/2+1$ vertices with weight $\zratio_{k\ell}$ with $\ell<k$.  Since the  $t_{k\ell}$ variables converge to a finite constant,  up to a constant multiple
 \[  wt(\gamma^1,\dotsc, \gamma^{m+1}) \le  C \prod_{\ell=1}^{m+1} e^{M(\ell-1)/2}
\cdot (e^{ -M/2})^{m(m+1)/2+1} \le  C e^{ -M/2}. \]
There is a fixed finite number of these $(m+1)$-tuples in \eqref{tau-2}, and so $\tau^{m+1}_k \le  C e^{ -M/2}$ for   $k>m+1$.

Next we establish a lower bound for $\tau^m_k$.  The $m$ rows of $d$-weights at the bottom of the array allow an $m$-tuple of paths that uses exactly $m(m-1)/2$ vertices with weight $\zratio_{k\ell}$ with $\ell<k$ (Figure \ref{NYfig8}). This $m$-tuple gives a positive constant lower bound:  $\tau^m_k\ge c>0$.

\begin{figure}
\begin{center}
\begin{picture}(200,140)(0,-15)
\multiput(0,0)(0,15){5}{\line(1,0){195}}
\put(202, -3){$d^{[1]}$}
\put(202, 12){$d^{[2]}$}   \put(204, 27){$\vdots$}
\put(202, 42){$d^{[m]}$}
\put(202, 57){$\zratio_1(0)$} \put(202, 72){$\zratio_2(0)$}  \put(207, 93){$\vdots$}
 \put(202, 117){$\zratio_{m+1}(0)$}

\put(15,75){\line(1,0){180}}
\put(30,90){\line(1,0){165}}  \put(45,105){\line(1,0){150}}\put(60,120){\line(1,0){135}}

\put(0,0){\line(0,1){60}}  \put(15,0){\line(0,1){75}}  \put(30,0){\line(0,1){90}} \put(45,0){\line(0,1){105}}
\multiput(60,0)(15,0){10}{\line(0,1){120}}

\multiput(-2,58)(15,15){4}{\large$\bullet$}
\put(-2,68){$1$} \put(13,83){$2$} \put(20, 96){$\iddots$}
 \put(42,113){$m$} \put(100, 128){$\dots$} \put(145, 128){$\dots$}  \put(189,125){$N$}

\multiput(57,-2)(15,0){4}{\large$\bullet$}
\put(-2,-10){$1$}  \put(15, -8){$\dots$}
\put(40,-10){$k-m+1$}  \put(85, -8){$\dots$}  \put(102,-10){$k$}   \put(143, -8){$\dots$}
 \put(189,-10){$N$}
{\linethickness{1.0pt}\put(-30,100){\vector(1,0){20}}
\put(-33,97){\vector(0,-1){20}}  }
{\linethickness{2.5pt}
\put(1,60){\line(0,-1){60}} \put(0,1){\line(1,0){61}}
\put(16,75){\line(0,-1){60}} \put(15,15){\line(1,0){61}}    \put(75,15){\line(0,-1){15}}
\put(31,90){\line(0,-1){60}} \put(30,30){\line(1,0){61}}    \put(90,30){\line(0,-1){30}}
\put(46,105){\line(0,-1){60}} \put(45,45){\line(1,0){61}}    \put(105,45){\line(0,-1){45}}
}
\end{picture}
\caption{An $m$-tuple of  paths for the minor $\tau^{m}_k$.}  \label{NYfig8}
\end{center} 
\end{figure}
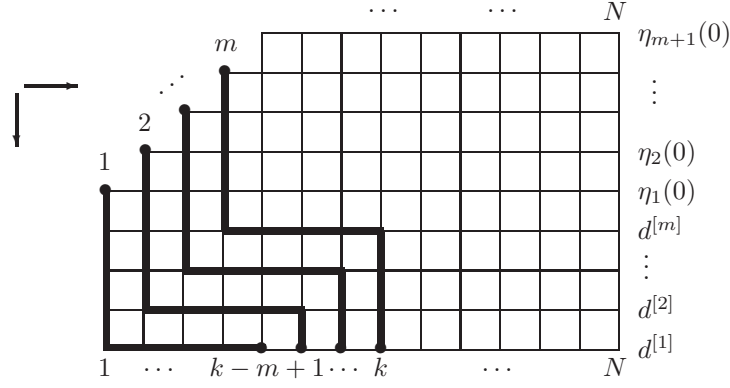

Combination of the first two bounds gives
\be    z_{k,m+1}(m) = \frac{\tau^{m+1}_{k}}{\tau^{m}_{k}} \le  C e^{ -M/2} \to 0
 \quad\text{ for $m+1< k\le N$.}
 \label{x10}\ee

It remains to consider the case $k=m+1$.  Minor $\tau^{m+1}_{m+1}$ has a
unique admissible  $(m+1)$-tuple, namely $m+1$ vertical paths. Consequently
\begin{align*}
\tau^{m+1}_{m+1} = \prod_{1\le\ell\le k\le m+1} \zratio_{k\ell} \cdot
\prod_{\substack{1\le i\le m\\ 1\le j\le m+1}} d_{i,j}
\; = \; e^{\sum_{\ell=1}^{m+1} t_{m+1,\ell} } \cdot \prod_{\substack{1\le i\le m\\ 1\le j\le m+1}} d_{i,j}.
\end{align*}
For $\tau^m_{m+1}$ the $m$-tuple in Figure \ref{NYfig8} has weight
\[   \prod_{1\le\ell\le k\le m} \zratio_{k\ell} \cdot
\prod_{\substack{1\le i\le m\\ 1\le j\le m+1}} d_{i,j}
\; = \; e^{\sum_{\ell=1}^{m} t_{m,\ell} } \cdot \prod_{\substack{1\le i\le m\\ 1\le j\le m+1}} d_{i,j}.\]
This $m$-tuple is minimal in its use of weights $\zratio_{k\ell}$ with $\ell<k$.  Other admissible $m$-tuples for  $\tau^m_{m+1}$ necessarily use more of such weights and
consequently pick up more $e^{-M/2}$ factors.   From all this
\be \begin{aligned}   z_{m+1,m+1}(m) &= \; \frac{\tau^{m+1}_{m+1}}{\tau^{m}_{m+1}}
\;=\; \frac{ e^{\sum_{\ell=1}^{m+1} t_{m+1,\ell} } \cdot \prod_{\substack{1\le i\le m\\ 1\le j\le m+1}} d_{i,j} }
{ e^{\sum_{\ell=1}^{m} t_{m,\ell} } \cdot \prod_{\substack{1\le i\le m\\ 1\le j\le m+1}} d_{i,j}
+ O(e^{-M/2})  }\\
&\underset{M\to\infty,\,t\to t^0}\longrightarrow  \;
\frac{ e^{\sum_{\ell=1}^{m+1} t^0_{m+1,\ell} } }{ e^{\sum_{\ell=1}^{m} t^0_{m,\ell} }}
\;=\; \frac{e^0}{e^0}\;=\; 1.
\end{aligned}  \label{x11}\ee
\eqref{x10} and \eqref{x11} together verify \eqref{M-lim13} and complete the proof of Lemma \ref{M-lim13-lm}.
\end{proof}


\begin{thebibliography}{}

\bibitem{AldervanMoerbeke}
M.~Adler, P.~van Moerbeke.
\newblock PDEs for the joint distributions of the Dyson, Airy and Sine processes.
\newblock {\it Ann. Probab.}, {\bf 33}:1326--1361 (2005).

\bibitem{AD}
D.~Aldous, P.~Diaconis.
\newblock Longest increasing subsequences: From patience sorting to the Baik-Deift-Johansson theorem.
\newblock {\it Bull. Amer. Math. Soc.}, {\bf 36}: 413--432 (1999).

\bibitem{BDJ}
J.~ Baik, P.A.~Deift, K.~Johansson.
\newblock On the distribution of the length of the longest increasing subsequence of random permutations.
\newblock {\it J. Amer. Math. Soc.}, {\bf 12}:1119--1178 (1999).

\bibitem{BCS}
M.~Bal\'{a}zs, E.~Cator, T.~Sepp\"{a}l\"{a}inen.
\newblock Cube root fluctuations for the corner growth model associated to the exclusion process.
\newblock {\it Electron. J. Probab.}, {\bf 11}:1094--1132 (2006).

\bibitem{BalaS}
M.~Bal\'{a}zs, T.~Sepp\"{a}l\"{a}inen.
\newblock Order of current variance and diffusivity in the asymmetric simple exclusion process.
\newblock {\it Ann. of Math.}, {\bf 171}:1237--1265 (2010).

\bibitem{BOCon}
F.~ Baudoin, N. ~O'Connell.
\newblock Exponential functionals of Brownian motion and class one Whittaker functions.
\newblock {\it Ann. Inst. H. Poincar\'{e} B}, in press.


\bibitem{BK}
A.~Berenstein, D.~Kazhdan.
\newblock Geometric and unipotent crystals.
\newblock Visions in Mathematics, Modern Birkh\"{a}user Classics, 188--236, 2010.

\bibitem{BK2}
A.~Berenstein, D.~Kazhdan.
\newblock Lecture notes on geometric crystals and their combinatorial analogues.
\newblock Combinatorial aspect of integrable systems, MSJ Memoirs, 17, Mathematical Society of Japan, Tokyo, 2007

\bibitem{bki}
A. Berenstein, A.N. Kirillov.
\newblock The Robinson-Schensted-Knuth bijection, quantum matrices and piece-wise linear combinatorics.
\newblock {\em Proceedings of 13th International Conference on Formal Power Series and Algebraic Combinatorics, }
Arizona State University, May 20-26, 2001.


\bibitem{BBO1}
Ph. ~Biane, Ph.~ Bougerol, N. ~O'Connell.
\newblock Littelmann paths and Brownian paths.
\newblock {\em Duke Math. J.},  {\bf 130}:127--167 (2005).

\bibitem{BBO2}
Ph. Biane, Ph. Bougerol, N. O'Connell.
\newblock Continuous crystals and Duistermaat-Heckman measure for Coxeter groups.
\newblock {\em Adv. Math.}, {\bf 221}:1522-1583 (2009).

\bibitem{BorCor}
A.~Borodin, I.~Corwin
\newblock Macdonald processes.
\newblock {\it Probab. Th. Rel. Fields}, to appear. arXiv:1111.4408.

\bibitem{BCR}
A.~Borodin, I.~Corwin, D.~Remenik.
\newblock Log-Gamma polymer free energy fluctuations via a Fredholm determinant identity.
\newblock {\em Commun. Pure Appl. Math.}, to appear. arXiv:1206.4573.

\bibitem{BP07}
A.~Borodin, S.~P\'ech\'e.
\newblock Airy kernel with two sets of parameters in directed percolation and random matrix theory.
\newblock {\it J. Stat. Phys.}, {\bf 132}:275--290 (2008).

\bibitem{BJ}
Ph.~Bougerol, Th.~Jeulin.
\newblock Paths in Weyl chambers and random matrices.
\newblock {\it Probab. Th. Rel. Fields}, {\bf 124}:517--543 (2002).

\bibitem{bump}
D. Bump.
\newblock {\it Automorphic forms on GL(3,${\R})$.}
\newblock Lecture Notes in Mathematics, 1083. Springer-Verlag, Berlin, 1984.

\bibitem{B}
D. Bump.
\newblock The Rankin-Selberg method: a survey, in Number Theory, Trace Formulas, and Discrete Groups (K. E. Aubert, E. Bombieri and D. Goldfeld,
eds.). Academic Press, New York, 1989.

\bibitem{CatG}
E.~Cator, P.~Groeneboom.
\newblock  Second class particles and cube root asymptotics for {H}ammersley's process
\newblock {\it Ann. Probab.}, {\bf 34}:1273--1295, (2006).

\bibitem{ICreview}
I.~Corwin.
\newblock The Kardar-Parisi-Zhang equation and universality class.
\newblock {\it Random Matrices: Theory and Appl.} {\bf 1}:1130001 (2012)

\bibitem{ch}
I. Corwin, A. Hammond.
\newblock The $H$-Brownian Gibbs property of the KPZ line ensemble.
\newblock In preparation.

\bibitem{dk}
V. Danilov, G. Koshevoy.
\newblock The octahedron recurrence and RSK-correspondence.
\newblock {\em Seminar Lotharinngien de Combinatoire}, B54An (2007).

\bibitem{Defosseux}
M.~Defosseux.
\newblock Orbit measures, random matrix theory and interlaced determinantal processes.
\newblock {\it Ann. Inst. H. Poincar\'{e} B}, {\bf 46}:209--249 (2010).

\bibitem{DW08}
A.B.~Dieker, J.~Warren.
\newblock On the largest-eigenvalue process for generalized Wishart random matrices
\newblock {\it ALEA}, {\bf 6}:369--376 (2009).

\bibitem{Doumerc}
Y.~Doumerc,
\newblock A note on representations of eigenvalues of classical Gaussian matrices.
\newblock S\`{e}minaire de Probabilit\'{e}s XXXVII. {\it Lecture Notes in Math.}, {\bf 1832}:370-384 (2003).

\bibitem{DMO}
 M.~Draief, J.~Mairesse, N. O'Connell.
\newblock Queues, Stores, and Tableaux.
\newblock {\it J. Appl. Probab.}, {\bf 42}:1145--1167 (2005).

\bibitem{FSreview}
P.L.~Ferrari, H.~Spohn.
\newblock Random growth models.
\newblock arXiv:1003.0881.

\bibitem{ForresterRains}
P.~J.~Forrester, E.~M.~Rains.
\newblock Jacobians and rank 1 perturbations relating to unitary Hessenberg matrices.
\newblock {\it Int. Math. Res. Not.}, 48306 (2006).

\bibitem{GKLO}
A. Gerasimov, S. Kharchev, D. Lebedev, S. Oblezin.
\newblock On a Gauss-Givental representation of quantum Toda chain wave equation.
\newblock {\it Int. Math. Res. Notices} 1--23 (2006).

\bibitem{GLO}
A. Gerasimov, D. Lebedev, S. Oblezin.
\newblock Baxter Operator and Archimedean Hecke Algebra.
\newblock {\it Commun. Math. Phys.} {\bf 284}:867--896 (2008).

\bibitem{givental}
A. Givental.
\newblock Stationary phase integrals, quantum Toda lattices, flag manifolds and the mirror conjecture.
\newblock {\it Topics in Singularity Theory}, AMS Transl. Ser. 2, vol. 180, AMS, Rhode Island (1997) 103--115.

\bibitem{GoldKont}
D.~Goldfeld, A.~Kontorovich.
\newblock On the determination of the Plancherel measure for Lebedev-Whittaker transforms on $GL(n)$.
\newblock arXiv:1102.5086.

\bibitem{g}
C. Greene.
\newblock An extension of Schensted's theorem.
\newblock {\em Adv. Math.} {\bf 14}:254--265 (1974).

\bibitem{KJ}
K.~Johansson.
\newblock Shape fluctuations and random matrices.
\newblock {\it Comm. Math. Phys.}, {\bf 209}:437--476 (2000).

\bibitem{KJ2}
K.~Johansson.
\newblock  Discrete orthogonal polynomial ensembles and the {P}lancherel measure
\newblock {\it  Ann. of Math.}, {\bf 153}:259--296 (2001).

\bibitem{KL}
S. Kharchev, D. Lebedev.
\newblock Integral representations for the eigenfunctions of quantum open and periodic Toda chains from the QISM formalism.
\newblock {\it J. Phys. A}, {\bf 34}:2247--2258 (2001).

\bibitem{Kir}
A. N. Kirillov.
\newblock Introduction to tropical combinatorics.
\newblock In: {\em Physics and Combinatorics. Proc. Nagoya 2000 2nd Internat.Workshop} (A. N. Kirillov and N. Liskova, eds.), World Scientific, Singapore, 82--150, 2001,

\bibitem{KOR}
W. K\"onig, N. O'Connell, S. Roch.
\newblock Non-colliding random walks, tandem queues, and discrete orthogonal polynomial ensembles.
\newblock {\em Electron. J. Probab.} {\bf 7} (2002)

\bibitem{k}
B. Kostant.
\newblock Quantisation and representation theory.  In: {\em Representation Theory  of Lie Groups}, Proc. SRC/LMS Research Symposium, Oxford 1977, LMS Lecture Notes 34,
Cambridge University Press, 1977, pp. 287--316.

\bibitem{LS}
B.~F.~Logan, L.~A.~Shepp.
\newblock A variational problem for random Young tableaux.
\newblock {\it Adv. Math.}, {\bf 26}:206--222 (1977).

\bibitem{m} I. Macdonald.
\newblock {\em Symmetric Functions and Hall Polynomials.}
\newblock Second Edition, Oxford University Press, 1995.

\bibitem{MY}
H. Matsumoto, M. Yor.
\newblock A version of Pitman's $2M-X$ theorem for geometric Brownian motions.
\newblock {\it C. R. Acad. Sci. Paris} {\bf 328}:1067--1074 (1999).

\bibitem{MO}
J. Moriarty, N. O'Connell.
\newblock On the free energy of a directed polymer in a Brownian environment.
\newblock {\it Markov Process. Rel. Fields} {\bf 13}:251-266 (2007).

\bibitem{NY}
M.~Noumi, Y.~Yamada.
\newblock Tropical Robinson-Schensted-Knuth correspondence and birational Weyl group actions.
\newblock {\it Representation theory of algebraic groups and quantum groups}, 371--442, Adv. Stud. Pure Math., 40, Math. Soc. Japan, Tokyo, 2004.

\bibitem{OCon}
N.~O'Connell.
\newblock Directed polymers and the quantum Toda lattice
\newblock  {\em Ann. Probab.}, {\bf 40}:437--458 (2012).

\bibitem{OCon2}
N. O'Connell.
\newblock Conditioned random walks and the RSK correspondence.
\newblock {\em J. Phys. A},  {\bf 36}:3049--3066 (2003).

\bibitem{OCon3}
N. O'Connell.
\newblock A path-transformation for random walks and the Robinson-Schensted
correspondence.
\newblock {\em Trans. Amer. Math. Soc.}, {\bf 355}:3669--3697 (2003).

\bibitem{OCon4}
N. O'Connell.
\newblock Random matrices, non-colliding processes and queues.
\newblock Seminaire de Probabilites XXXVI, 165--182.  Lecture Notes in Mathematics 1801, Springer, 2002.

\bibitem{OSZ}
N.~O'Connell, T.~Sepp\"{a}l\"{a}inen, N.~Zygouras.
\newblock Geometric RSK correspondence, Whittaker functions and symmetrized random polymers.
\newblock arXiv:1210.5126.


\bibitem{ow}
N.~O'Connell, J.~Warren.
\newblock A multi-layer extension of the stochastic heat equation.
\newblock arXiv:1104.3509.

\bibitem{OConYor}
N.~O'Connell, M.~Yor.
\newblock  Brownian analogues of {B}urke's theorem.
\newblock  {\it Stochastic Process. Appl.}, {\bf 96}:285--304 (2001).

\bibitem{OConYor2}
N. O'Connell, M. Yor.
\newblock A representation for non-colliding random walks.
\newblock {\em  Electron. Comm. Probab.},  {\bf 7} (2002).

\bibitem{OkInfWedge}
A.~Okounkov.
\newblock Infinite wedge and random partitions
\newblock {\it Selecta Math.},{\bf 7}:57--81 (2001).

\bibitem{R}
K. Rietsch.
\newblock A mirror construction for the totally nonnegative part of the Peterson variety.
\newblock {\em Nagoya Math. J.} {\bf 183}:105--142 (2006).

\bibitem{PR}
L.C.G.~Rogers, J.W.~Pitman.
\newblock Markov functions.
\newblock {\em Ann. Probab.}, {\bf 9}:573--582 (1981).

\bibitem{ss}
T. Sasamoto, H. Spohn.
\newblock The $1+1$-dimensional {K}ardar-{P}arisi-{Z}hang equation and its universality class.
\newblock Proceedings StatPhys 24; Journal of Statistical Mechanics, online (2010).

\bibitem{sts}
M. Semenov-Tian-Shansky.
\newblock Quantisation of open Toda lattices.
\newblock In: {\em Dynamical systems VII: Integrable systems, nonholonomic dynamical systems. }
Edited by V. I. Arnol'd and S. P. Novikov. Encyclopaedia of Mathematical Sciences, 16. Springer-Verlag, 1994.

\bibitem{S98}
T.~Sepp\"{a}l\"{a}inen.
\newblock Hydrodynamic scaling, convex duality and asymptotic shapes of growth models.
\newblock {\it Markov Process. Rel. Fields}, {\bf 4}:1--26 (1998).

\bibitem{S99}
T.~Sepp\"{a}l\"{a}inen.
\newblock  Exact limiting shape for a simplified model of first-passage percolation on the plane.
\newblock  {\it Ann. Probab.}, {\bf 26}:1232--1250 (1999).

\bibitem{S}
T.~Sepp\"{a}l\"{a}inen.
\newblock Scaling for a one-dimensional directed polymer with boundary conditions.
\newblock  {\it Ann. Probab.}, {\bf 40}:19--73 (2012).

\bibitem{SV}
T.~Sepp\"{a}l\"{a}inen, B.~Valko.
\newblock Bounds for scaling exponents for a 1+1 dimensional directed polymer in a Brownian environment.
\newblock {\it ALEA Lat. Am. J. Probab. Math. Stat.}, {\bf 7}:451--476 (2010).

\bibitem{St} E. Stade.
\newblock Archimedean $L$-factors on $GL(n)\times GL(n)$ and generalized Barnes integrals.
\newblock {\it Israel J. Math.} {\bf 127}:201--219 (2002).

\bibitem{TW}
C.~Tracy and H.~Widom.
\newblock Level-spacing distributions and the Airy kernel.
\newblock {\em Comm. Math. Phys.}, {\bf 159}:151--174 (1994).

\bibitem{VK}
A.~M.~Vershik, S.~Kerov.
\newblock Asymptotics of the Plancherel measure of the symmetric group and the limiting form of Young tables.
\newblock {\it Soviet Math. Dokl.}, {\bf 18}:527--531 (1977).

\end{thebibliography}
\end{document}